\documentclass[10pt]{article} 
\pdfoutput=1
\usepackage{fullpage}
\usepackage{changepage}
\usepackage[lined,ruled,boxed,linesnumbered]{algorithm2e}
\usepackage{array}
\usepackage{graphics} 
\usepackage{latexsym}
\usepackage{subfigure} 
\usepackage{siunitx}
 \usepackage{tikz}
\usepackage{framed}
\usepackage{pgfplots}
\usepackage{ifpdf}
\usetikzlibrary{spy,backgrounds}
\iftrue
  \usepackage{tikz}
  \usepackage{pgfplots}
  \pgfplotsset{
    compat=1.3,
    tick label style={font=\scriptsize},
    label style={font=\scriptsize},
    legend style={font=\scriptsize}
  }
  
\pgfplotsset{
 compat=1.3, 
 tick label style={font=\small},
 label style={font=\small},
 legend style={font=\footnotesize},
contour/draw color={black},
contour/labels=false
}
  
  \usepgfplotslibrary{external}
\makeatletter
\newcommand{\gettikzxy}[3]{%
  \tikz@scan@one@point\pgfutil@firstofone#1\relax
  \edef#2{\the\pgf@x}%
  \edef#3{\the\pgf@y}%
}
\makeatother
\usepgfplotslibrary{groupplots}
\usetikzlibrary{shapes.geometric}
\usetikzlibrary{spy}
\usetikzlibrary{snakes}
\usetikzlibrary{backgrounds}
 \usepgfplotslibrary{external}
 
\else
  \usepackage{tikzexternal}
  \tikzsetexternalprefix{fig/}
  
\fi

\usepackage{mathtools}
\usepackage{enumerate}
\usepackage{amssymb}
\usepackage[mathscr]{eucal}

 \usepackage{amsthm}
 \usepackage[utf8x]{inputenc}
\usepackage{bbm}
\usepackage{bm}
\usepackage{color}

\usepackage{mathtools}

\mathtoolsset{showonlyrefs}

\newenvironment{myenv}{\begin{adjustwidth}{1cm}{}}{\end{adjustwidth}}
 \newtheorem{thm}{Theorem}[section]
 \newtheorem{cor}[thm]{Corollary}
 \newtheorem{lem}[thm]{Lemma}
 \newtheorem{ass}[thm]{Assumption}
 \newtheorem{prop}[thm]{Proposition}
 \newtheorem{ex}{Example}

\newtheorem{rem}{Remark} 

\newcommand{\malph}{\langle \alpha\rangle}
\newcommand{\mbet}{\Big \langle \dfrac{\beta}{v}\Big\rangle}
\newcommand{\Fact}{\varepsilon^\gamma}
\newcommand{\Oh}{O}
\newcommand{\oh}{o}
\renewcommand{\d}{\mathrm{d}}

\newcommand{\R}{\mathbb{R}}

\newcommand{\J}{\mathcal{J}}

\newcommand{\eps}{\varepsilon}     
 \newcommand{\diag}{\,\mathrm{diag}} 
\newcommand*\samethanks[1][\value{footnote}]{\footnotemark[#1]}

\newcommand{\step}{\mathcal{S}_{\delta t}}

\newcommand{\ut}{\tilde{u}}
\newcommand{\etildeexpr}{\Id+\varepsilon^\gamma V^{-1}}
\newcommand{\Id}{\mathcal{I}}
\newcommand{\deriv}{\mathscr{D}}
\newcommand{\Maxw}{{\mathcal{M}_v}}
\newcommand{\Maxwj}{{\mathcal{M}_{v_j}}}
\newcommand{\imag}{\imath}
\newcommand{\feps}{f^{\varepsilon}}
\newcommand{\ueps}{u^{\varepsilon}}

\renewcommand{\Re}{\mathrm{Re}}
\renewcommand{\Im}{\mathrm{Im}}

\title{A high-order asymptotic-preserving scheme for kinetic
equations using projective integration}
\author{Pauline Lafitte\thanks{Laboratoire de
    Math\'ematiques Appliqu\'ees aux Syst\`emes, Ecole Centrale Paris, Grande
    Voie des Vignes, 92290 Ch\^atenay-Malabry, France ({\tt pauline.lafitte@ecp.fr}).} \and Annelies Lejon\thanks{Department of Computer Science, KU Leuven, Celestijnenlaan
200A, 3001 Leuven, Belgium (firstname.lastname@cs.kuleuven.be). The second
author's work was supported by the Agency for Innovation by Science and
Technology in Flanders (IWT) } \and Giovanni Samaey\samethanks{} }
\date{} 

\begin{document}
\maketitle
\begin{abstract}
  We investigate a high-order, fully explicit, asymptotic-preserving scheme for
  a kinetic equation with linear relaxation, both in the hydrodynamic and
  diffusive scalings in which a hyperbolic, resp.~parabolic, limiting equation
  exists. The scheme first takes a few small (inner) steps with a simple,
  explicit method (such as direct forward Euler) to damp out the stiff
  components of the solution and estimate the time derivative of the slow
  components. These estimated time derivatives are then used in an (outer)
  Runge--Kutta method of arbitrary order.  We show that, with an appropriate
  choice of inner step size, the time-step restriction on the outer time step is
  similar to the stability condition for the limiting macroscopic
  equation. Moreover, the number of inner time steps is also independent of the
  scaling parameter.  We analyse stability and consistency, and illustrate
  with numerical results.
\end{abstract}

%
%
%
\section{Introduction}


In many applications (such as traffic flow, biology or
physics), the system under study consists of a large number of interacting particles. One option is to simulate such systems at a microscopic level, via an agent-based description with great modelling detail. 
At a mesoscopic level, one can write a kinetic description that governs the evolution of the particle distribution in position-velocity space. Then, $f(x,v,t)$ represents the probability of finding a particle at position $x$, moving with velocity $v$ at time $t$.  Its evolution is governed by a kinetic equation,
\begin{equation}
    \partial_t \feps + v\partial_x\feps =Q(\feps),
    \label{eq:intro_kineq}
\end{equation}
in which the lefthand side describes free transport and $Q(f)$ embodies collisions (velocity changes). 
Equation \eqref{eq:intro_kineq} can be made dimensionless via a rescaling with
respect to the characteristic length $L$, time $T$ and velocity $V$
scales
\begin{equation*}
    \tilde{x} = L x \qquad \tilde{t} = T t \qquad
    \tilde{v} = V v.
\end{equation*}
The regimes in which we are interested are $L=VT
\varepsilon^\gamma$, where $\varepsilon$ is a positive constant and $\gamma$ is an integer that indicates a hydrodynamic ($\gamma=0$) or
diffusive ($\gamma=1$) scaling.  (Details on the choice of scaling are in
section~\ref{sec:modprob}.) This, omitting the tildes, results in the dimensionless equation
\begin{equation}\label{eq:kin_eq_dimensionless}
    \partial_{t}
    \feps+\dfrac{v}{\varepsilon^\gamma}\partial_{x}\feps =
    \dfrac{Q(\feps)}{\varepsilon^{\gamma+1}}.
\end{equation}

 In a diffusive or hydrodynamic scaling, one can usually obtain an approximate
 macroscopic partial differential equation (PDE) for a number of low-order
 moments of the particle distribution $f$ (such as density, momentum, etc.).  
Improving upon the macroscopic approximation, however, is computationally expensive.
Because of the stiffness of~\eqref{eq:kin_eq_dimensionless}, explicit methods require an excessively small time-step for small values of $\varepsilon$, whereas implicit methods suffer from the high dimensionality of the problem. 

There is currently a large research effort in the design of algorithms that are
uniformly stable in $\varepsilon$ and approach a scheme for the limiting
equation when $\varepsilon$ tends to $0$; such schemes are called
\emph{asymptotic-preserving in the sense of Jin} \cite{Jin1999}. We briefly
review here some achievements, and refer to the cited references for more
details. In \cite{jipato1,jipato2,klar98}, separating the distribution $f$ into
its odd and even parts in the velocity variable results in a coupled system of
transport equations where the stiffness appears only in the source term,
allowing to use a time-splitting technique \cite{strang} with implicit treatment
of the source term; see also related work in \cite{Jin1999, klar98, klar98-2,
  klar99}.  Implicit-explicit (IMEX) schemes are an extensively studied
technique to tackle this kind of problems \cite{Ascher1997,Filbet2010} (and
references therein). Recent results in this setting were obtained by Dimarco
\emph{et al.}~to deal with nonlinear collision kernels \cite{Dimarco2013}, and
an extension to hyperbolic systems in a diffusive limit is given in
\cite{Boscarino2013}. 
A different point of view based on well-balanced methods was introduced by Gosse
and Toscani \cite{goto1,goto2}, see also \cite{Buet2006717,buco}. Discontinuous
Galerkin schemes have also been developed \cite{lamomi, adam, lomo, lomc, guka},
as well as regularization methods \cite{halo,haha}. When the collision operator
allows for an explicit computation, an explicit scheme can be obtained subject
to a classical diffusion CFL condition by splitting $f$ into its mean value and
the first-order fluctuations in a Chapman-Enskog expansion form
\cite{gl1}. Also, closure by moments \cite[e.g.]{cogogo} can lead to reduced
systems for which time-splitting provides new classes of schemes
\cite{Carrillo2008}, see \cite{mine, pomr, lemi,struchtrup2005macroscopic} for
more complete references on moment methods in general. Alternatively, a
micro-macro decomposition based on a Chapman-Enskog expansion has been proposed
\cite{lemi}, leading to a system of transport equations that allows to design a
semi-implicit scheme without time splitting. An non-local procedure based on the
quadrature of kernels obtained through pseudo-differential calculus was proposed
in \cite{besse2010derivation}.



In \cite{lafitte2012}, an alternative, fully explicit, asymptotic-preserving
method was proposed, based on projective integration, which was introduced in
\cite{gear:1091} as an explicit method for stiff systems of ordinary
differential equations (ODEs) that have a large gap between their fast and slow
time scales; these methods fit within recent research efforts on numerical
methods for multiscale simulation
\cite{EEng03,E:2007p3747,KevrGearHymKevrRunTheo03,Kevrekidis2009a}; see
also~\cite{logg,sommeijer1990increasing,Vandekerckhove2007a} for related
approaches. In projective integration, the fast modes, corresponding to the
Jacobian eigenvalues with large negative real parts, decay quickly, whereas the
slow modes correspond to eigenvalues of smaller magnitude and are the main
contributions to the solution. Projective integration allows a stable yet
explicit integration of such problems by first taking a few small (inner) steps
with a simple, explicit method, until the transients corresponding to the fast
modes have died out, and subsequently projecting (extrapolating) the solution
forward in time over a large (outer) time step; a schematic representation of
the scheme is given in figure~\ref{fig:PI}.
 \begin{figure}[h]
\centering
    \includegraphics[width=0.45\textwidth]{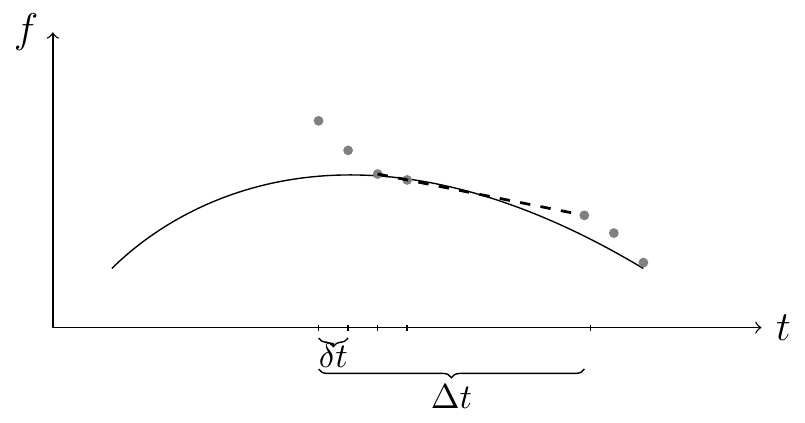}
    \caption{Illustration of first order projective integration.}
    \label{fig:PI}
\end{figure}
In \cite{lafitte2012}, this method was shown to be asymptotic-preserving for kinetic equations in the diffusive scaling with a linear relaxation collision operator: given an adequate choice of the size of the inner time step, one can obtain a method that has a CFL-type step-size restriction on the outer time step, and requires a number of inner steps that is independent of $\varepsilon$.  The computational cost of the method is thus independent of $\varepsilon$. 

%

Many of the above-described methods have inherent limitations with respect to
the order that can be achieved with the time discretisation, for instance due to
the time-splitting. In this paper, we present a projective integration method
that allows to attain arbitrary order accuracy in time. The generalisation is
based on a modification of classical Runge--Kutta methods, and retains all
advantages of the method in \cite{lafitte2012}, i.e., it is fully explicit and asymptotic-preserving. 
Additionally, we significantly extend the analysis of the scheme.  Specifically,
the results in \cite{lafitte2012} are limited to the diffusive scaling, and to
an equation that has a pure diffusion limiting equation. In this
paper, we extend these results to model equations that result in an
advection-diffusion limit when $\varepsilon$ tends to $0$, and the analysis now covers both the diffusive and the hydrodynamic scaling.  
In \cite{Melis2013}, we discuss and illustrate how the
method can be used in conjunction with a relaxation method \cite{arna,JinXin95}
to create a fully general, explicit time integration method for hyperbolic
systems of conservation laws, also in multiple space dimensions. We remark that
alternative approaches to obtain a higher-order projective integration scheme have been proposed in \cite{Lee2007,Rico-Martinez}.

The remainder of the paper is organized as follows.
In section~\ref{sec:modprob}, we discuss the model problems that will be used
during the numerical experiments. We then discuss the
projective Runge-Kutta method (PRK) in section~\ref{sec:PRK}, and provide a result concerning its stability region. 
In section~\ref{sec:spectrum}, we perform 
an analysis of the spectrum of the kinetic equations introduced in section~\ref{sec:modprob}, generalising the results obtained in \cite{lafitte2012} to the hydrodynamic scaling and to systems with macroscopic advection.  The analysis also reveals how to choose the different method parameters of the projective Runge--Kutta method. In
section~\ref{sec:consist} we give a consistency proof that shows the order of accuracy. We illustrate the results with some numerical experiments in section~\ref{sec:exp}.  Finally, section~\ref{sec:concl} contains a conclusion and outlook to future work.

\section{Model problems}\label{sec:modprob}
\subsection{A simple kinetic equation}

As a first model problem, we study a dimensionless scalar kinetic equation with linear relaxation
in one space dimension,
\begin{equation}
		\partial_t f^\varepsilon+\dfrac{v}{\varepsilon^\gamma}\partial_x
		f^\varepsilon=\dfrac{\Maxw(u^\varepsilon)-f^\varepsilon}{\varepsilon^{\gamma+1}}\label{kin_eq},
\end{equation}
modelling the evolution of a particle distribution function
$f^\varepsilon(x,v,t)$ that gives the distribution of particles at a given
position $x\in U=[-1,1)$ with velocity $v\in V\subset \R$ at time $t>0$,
$\varepsilon$ being a positive fixed constant. For the consistency analysis, we
will impose periodic boundary conditions In the numerical experiments, we will
also use Neumann boundary conditions. The parameter $\gamma$ defines the
scaling: when $\gamma=0$, the scaling is called hydrodynamic; $\gamma=1$
corresponds to a diffusive scaling.  The righthand side represents a BGK
collision operator \cite{Bhatnagar1954} that models linear relaxation of
$f^\varepsilon$ towards a Maxwellian distribution $\Maxw(u^\varepsilon)$, in
which $u^\varepsilon(x,t)=\langle f^\varepsilon(x,v,t)\rangle$ is the density,
obtained via averaging over the measured velocity space $(V,\mu)$, i.e.,
\begin{equation}\label{eq:bracket}
	\langle \cdot \rangle = \int_V \cdot \; d\mu(v).
\end{equation}

Let us now discuss the measured velocity space $(V,\mu)$ and the Maxwellian $\Maxw$.

\paragraph{Velocity space} We consider odd-symmetric velocity spaces $(V,\mu)$ :
\begin{equation}
\begin{cases}
\langle 1 \rangle=\int_V \d \mu(v)=1,\\
\langle h \rangle=\int_V h(v)\d \mu(v)=0\text{ for any odd integrable function $h: V \longrightarrow \R$,}\\
\langle v^2 \rangle=\int_V v^2\d \mu(v)=d>0.
\end{cases}
\end{equation}
We restrict ourselves to discretized velocity spaces
of the form
\begin{equation}\label{eq:defV}
V:=\{ v_j \}_{j=1}^J, \qquad d\mu(v)=\sum_{j=1}^J w_j \delta(v-v_j),
\end{equation}
with $J$ even, where the velocities satisfy $v_{j}=-v_{J-j}$ for all $j$, and $w_j$ are
appropriately chosen weights that satisfy $\sum_j^J w_j=1$. 

In the diffusive scaling ($\gamma=1$), the discrete velocity space $V$ results
from applying a Gauss quadrature discretisation to~\eqref{eq:bracket}
\cite{cohen2011numerical}. Throughout the analysis, we will consider a uniform
symmetric discretisation of $V=(-1,1)$, i.e., ${v_{j} =(2j-J-1)/J}$, with ${J/2+1
\le j \le J}$; the weights are then defined as $w_j=1/J$. In our numerical
experiments, we will use discretisations of (i) the velocity space $V=(-1,1)$
endowed with the Lebesgue measure; and (ii) the velocity space $V = \R$ endowed
with the Gaussian measure
$\d \mu(v)=(2\pi)^{-1/2}\exp(-v^2/2)\d v$.  Then, $v_j$ are chosen
as the roots of the Legendre, resp.~Hermite, polynomial of degree $2J$, and the
$w_j$ are the corresponding quadrature weights.  In the hyperbolic scaling
$(\gamma=0)$, $(V,\mu)$ needs to satisfy the subcharacteristic condition (which
ensures the positivity of the diffusion coefficient), see, e.g.,
\cite{arna,Melis2013}.


\paragraph{Maxwellian} Let us assume that the Maxwellian $\Maxw$ 
satisfies (see, e.g., \cite{arna, Bouchut1998})
\begin{equation}\label{eq:max_conditions}
		\left\langle \Maxw(u)\right\rangle = u,\qquad 
		\left\langle v \Maxw(u) \right\rangle = \varepsilon^\gamma A (u).
\end{equation}
Throughout the analysis and numerical experiments, we will use
\begin{equation}\label{eq:maxw-proto}
	\Maxw(u) = u +\varepsilon^\gamma\dfrac{A(u)}{v}.
\end{equation}
In the analysis, we will restrict ourselves to the linear case, $A(u)=u$.

Let us now discuss the
limiting macroscopic equation when $\varepsilon$ tends to $0$ by performing a
Chapman-Enskog expansion,
\begin{equation}\label{eq:ce}
f^\varepsilon=\Maxw(u^\varepsilon)+\varepsilon g^\varepsilon,
\end{equation}
with $\langle g^\varepsilon\rangle=0$.
Substituting
\eqref{eq:ce}
into the model equation \eqref{kin_eq} yields
\begin{equation}\label{eq:filled-in}
 \partial_t\left(\Maxw(\ueps)+\varepsilon
 g^\varepsilon\right) +
 \dfrac{v}{\varepsilon^\gamma}\partial_x\left(\Maxw(\ueps)+\varepsilon
 g^\varepsilon\right) = -\dfrac{g^\varepsilon}{\varepsilon^\gamma}.
\end{equation}
Then, taking the mean over velocity
space and using \eqref{eq:max_conditions}, we obtain
\begin{equation}
	\partial_t \ueps+ \partial_x (A(\ueps)) + \varepsilon^{1-\gamma}\langle v\partial_x
	g^\varepsilon\rangle = 0.
\end{equation}
The last term on the lefthand side can be approximated by considering the terms in~\eqref{eq:filled-in} of
order $\Oh(1/\varepsilon^\gamma)$, from which we obtain
$g^\varepsilon=-v\partial_x \Maxw(\ueps)+\Oh(\varepsilon)$. This gives rise to
\begin{equation}
	\partial_t \ueps + \partial_x (A(\ueps)) = \varepsilon^{1-\gamma} d\;
	\partial_{xx}\ueps+\Oh(\varepsilon^2), \qquad d=\langle
	v^2\rangle \label{eq:diff_limit}
\end{equation}
Depending on the scaling, we thus obtain a hyperbolic advection equation
($\gamma=0$) or a parabolic advection-diffusion equation ($\gamma=1$) when
$\varepsilon$ tends to $0$.

In this paper, we will analyse the properties of the
projective integration method in both the parabolic and the hyperbolic scaling. The numerical experiments in the present paper focus on the parabolic scaling, in which equation~\eqref{kin_eq} becomes
\begin{equation}
		\partial_t f^\varepsilon+\dfrac{v}{\varepsilon}\partial_x
		f^\varepsilon=\dfrac{\Maxw(u^\varepsilon)-\feps}{\varepsilon^{2}}\label{kin_eq_diff},
\end{equation}
with macroscopic limit
\begin{equation}
	\partial_t \ueps + \partial_x (A(\ueps)) =  d\;
	\partial_{xx}\ueps.\label{eq:diff_limitpar}
\end{equation}

Besides linear advection, we will also consider the viscous Burgers'
equation, which is obtained when choosing $A(u)=u^2$.
Numerical examples in the hyperbolic scaling are given in \cite{Melis2013}, which also discusses the generalisation to multiple space
dimensions.

%
%
%

\subsection{A kinetic semiconductor equation}\label{subsec:semicond}

While the numerical analysis of the presented algorithms is restricted to the above kinetic equation with $A(u)$ linear,
we will also provide numerical results for a second model problem, in which macroscopic advection does not originate from
the Maxwellian in the collision operator, but from an external force field.  To this end, we consider a kinetic equation that
is inspired by the semiconductor equation \cite{ali},
\begin{align}
	\begin{aligned}
		&&\partial_t f^\varepsilon+\dfrac{1}{\varepsilon}\left(v\partial_x f^\varepsilon+F \partial_v f^\varepsilon\right)=\dfrac{u^\varepsilon-f^\varepsilon}{\varepsilon^2},\\
		&&F=-\nabla\cdot \Phi, \qquad \Delta \Phi=u^\varepsilon.
\end{aligned}\label{eq:semi}
\end{align}
This equation describes the evolution of the distribution function
$f^\varepsilon(x,v,t)$, in which now an acceleration also appears due to an
electric force $F$ resulting from a coupled Poisson equation for the electric potential $\Phi$.  The velocity space is given by $V = \R$ endowed with the Gaussian measure $\d \mu(v)=(2\pi)^{-1/2}\exp(-v^2/2)\d v$.

\section{High-order projective integration}\label{sec:PRK}

The algorithm we propose in this paper is a high-order Runge--Kutta extension
of the projective integration method \cite{gear:1091,lafitte2012},
which will turn out to be a fully explicit, arbitrary order,
asymptotic-preserving time integration method for the kinetic equation
\eqref{kin_eq}. The asymptotic-preserving property \cite{Jin1999} implies that,
in the limit when $\epsilon$ tends to zero, an $\varepsilon$-independent time
step constraint can be used, similar to the hyperbolic CFL-constraint for the
limiting equation \eqref{eq:diff_limit}, depending on the scaling of \eqref{kin_eq}. To achieve this, the projective integration method combines a few small time steps with a naive (\emph{inner}) time-stepping method
with a much larger (\emph{projective, outer}) time step.  The asymptotic-preserving property will then follow from the observation that both the size of the outer time step and the \emph{number} of inner steps are independent of $\varepsilon$, resulting in a total computational cost that is independent of $\varepsilon$.

In sections~\ref{sec:inner} and~\ref{sec:outer}, we discuss the inner and
outer integrators, respectively. We then discuss the stability regions of
the projective integration method in section~\ref{sec:stab}.
\subsection{Inner integrator}\label{sec:inner}

We intend to integrate \eqref{kin_eq} on a uniform, constant in time, periodic
spatial mesh with spacing $\Delta x$, consisting of $I$ mesh points $x_i=i\Delta
x$, $0\le i \le I$, with $I\Delta x=1$, and a uniform time mesh with time step
$\delta t$, i.e., $t^k=k\delta t$ and $k\geq 0$.  The numerical solution on this
mesh is denoted as $f_{i,j}^k$, where we have dropped the dependence on
$\varepsilon$ in the numerical solution for conciseness.
After discretising in space, we obtain a semi-discrete system of ordinary differential equations
\begin{equation}\label{eq:semidiscrete}
\dot{f} = \deriv_t(f),  \qquad \deriv_t(f):=-
\dfrac{1}{\varepsilon^\gamma}\deriv_{x,v}(f) +
\frac{1}{\varepsilon^{\gamma+1}}\left(\Maxw(u) - f\right),
\end{equation}
where $\deriv_{x,v}(\cdot)$ represents a suitable discretisation of the first
spatial derivative and $u = \langle f \rangle$.  In the parabolic case, central
differences are necessary (see \cite{lafitte2012} and the next sections) and in
the related numerical experiments, we will use a fourth order discretisation,
\begin{equation}
	(\deriv_{x,v_j} f)_{i,j} = \dfrac{-f_{i+2,j} +8f_{i+1,j} -8f_{i-1,j} +f_{i-2,j}
}{12\Delta x}. \label{eq:derivx-par}
\end{equation}
 In the hyperbolic case, some type of upwinding needs to be performed and we will use, in the numerical experiments, a third order upwind biased scheme,
\begin{equation}
	\begin{cases}
		(\deriv_{x,v_j} f)_{i,j} = v_j\dfrac{2f_{i+1,j} + 3f_{i,j} - 6f_{i-1,j}
		+f_{i-2,j}}{6\Delta x} \qquad\text{if $v_j>0$,}\\[0.5cm]
		(\deriv_{x,v_j} f)_{i,j} = v_j\dfrac{-f_{i+2,j} + 6f_{i+1,j}-3f_{i,j} -
		f_{i-1,j}}{6\Delta x}
		\qquad\text{if $v_j<0$.}
	\end{cases}\label{eq:derivx-hyp}
\end{equation}

Combined with a forward Euler time discretisation, we obtain
\begin{equation}
 f_{i,j}^{k+1}=f_{i,j}^{k}-\dfrac{\delta
t}{\varepsilon^\gamma}\deriv_{x,v_{j}}(f^{k})_{i,j}+\dfrac{\delta
t}{\varepsilon^{\gamma+1}}(
\Maxwj(u^{k}_i)-f^{k}_{i,j})\qquad \forall 1\le i\le I,\quad 1\le j\le J,\label{eq_disc_scheme}
\end{equation}
which we also denote using the shorthand notation
\begin{equation} \label{scheme_time-stepper} f^{k+1} =\mathcal{S}_{\delta t}(f^{k}),\qquad k = 0, 1, \ldots
\end{equation}
In the context of projective integration, it does not make sense to investigate higher order methods for the inner integration. Some remarks on this fact are made in \cite{Melis2013}.

\subsection{Outer integrator}\label{sec:outer}

The model problem we are dealing with is clearly stiff because of the presence
of the small Knudsen parameter $\varepsilon$, leading to a time step restriction
for the naive scheme \eqref{eq_disc_scheme} of $\Oh(\varepsilon^{\gamma+1})$ due to the
relaxation term. However, as $\varepsilon$ goes to $0$, we are able to obtain a
limiting equation for which a standard finite volume/forward Euler method only
needs to satisfy a stability restriction of the form $\Delta t \le C\Delta
x^{\gamma+1}$, with $C$ a constant that depends on the specific choice of the
scheme and the parameters of the equation.

In \cite{lafitte2012}, the projective integration technique was proposed to
accelerate brute force integration; the idea, originating from \cite{gear:1091}, is
the following.  Starting from a numerical
solution $f^N$ at time $t^N=N\Delta t$, one first takes $K+1$ \emph{inner} steps of size $\delta t$,
$f^{N,k+1}=S_{\delta t}(f^{N,k}),\quad k=0,\ldots,K$, in which the superscript
pair $(N,k)$ represents a numerical solution by means of the inner scheme at time $t^{N,k}=N\Delta t +k\delta t$. The aim is to obtain
a discrete derivative to be used in the \emph{outer} step to compute $f^{N+1}=f^{N+1,0}$ via extrapolation in time, e.g.,
\begin{equation}\label{e:pfe}
  f^{N+1} = f^{N,K+1} + \left(\Delta t-(K+1)\delta t\right)\frac{f^{N,K+1} -
    f^{N,K}}{\delta t}.
\end{equation}
This method is called projective forward Euler (PFE), and it is the simplest
instantiation of this class of integration methods
\cite{gear:1091,lafitte2012}.

In this paper, we present a particular higher order extension of this idea, based on Runge--Kutta methods.  Let us denote a general explicit $S$-stage Runge--Kutta method for equation~\eqref{kin_eq} with time step $\Delta t$ as
\begin{align}
&\begin{cases}	f_s^{N+c_s}&= f^N + \Delta t \sum_{l=1}^{s-1} a_{s,l} k_l\\
 k_s &= \deriv_t\left(f_s^{N+c_s}\right)
\end{cases},\qquad
 1 \le s \le S, \\
&f^{N+1} = f^N + \Delta t \sum_{s=1}^{S}b_s k_s,
 \end{align}
with $\deriv_t$ defined in~\eqref{eq:semidiscrete}.  As in \cite{Hairer1993}, we
call the matrix $A=(a_{s,l})_{s,l=2}^{S}$ the RK matrix, $b=(b_s)_{s=1}^S$ the
RK weights and $c=(c_s)_{s=1}^S$ the RK nodes.  The values $k_s$ are called the
RK stages, and represent an approximation of the time derivative at time
$t=t^N+c_s\Delta t$. The weights $b_s$ and $c_s$ are chosen
simultaneously, and correspond to a Gauss quadrature approximation of the
integration from $t=t^N$ to $t^{N+1}$. To ensure consistency, these coefficients
satisfy the following assumptions (see, e.g., \cite{Hairer1993}):
\begin{ass}[Runge--Kutta coefficients]\label{ass:RK}
The Runge--Kutta coefficients satisfy $0\le c_s \le 1$, resp.~ $0 \le b_s \le 1,$ and 
\begin{align}
			\sum_{s=1}^Sb_s=1,\qquad  \sum_{l=1}^{S-1} a_{s,l} =c_s, \quad 1 \le s \le S.&\label{eq:cond3}
\end{align}
(Note that these assumptions imply that $c_1=0$ by the convention that $\sum_{1}^0\cdot=0$).
\end{ass}

In the higher order projective integration method, we proceed, by analogy with
the projective forward Euler method, by replacing each time derivative
evaluation $k_s$ by $K+1$ steps of an inner integrator and a time derivative
estimate as follows (with $f^{N,0}=f^N$ for consistency):
\begin{align}
s = 1 : & \begin{cases} f^{N,k} &= f^{N,k-1} + \delta t \deriv_t(f^{N,k-1}), \qquad 1 \le k \le K+1 \\
k_1 &= \dfrac{f^{N,K+1}-f^{N,K}}{\delta t} ,
\end{cases}\label{eq:rk1}\\
2 \le s \le S :& \begin{cases}
f^{N+c_s}_s&=f^{N,K+1} + (c_s\Delta t-(K+1)\delta t) \sum_{l=1}^{s-1}
\dfrac{a_{s,l}}{c_s} k_l, \\
f^{N+c_s,k}&=  f^{N+c_s,k-1} + \delta t \deriv_t(f^{N+c_s,k-1}), \qquad 1 \le k \le K+1 \\
k_s &= \dfrac{f^{N+c_s,K+1}-f^{N+c_s,K}}{\delta t},
\end{cases}\label{eq:rk2}\\
& f^{N+1} = f^{N,K+1} + (\Delta t-(K+1)\delta t)\sum_{s=1}^{S}b_s
k_s.\label{eq:rk3}
\end{align}
In the following sections, it will be shown that the small time step should be
taken as 
\begin{equation}
\delta t=\eps^{\gamma+1}.\label{eq:dt}
\end{equation}

Note that the stages $k_s$ now record a finite difference approximation of the time derivative at
time $t=t^N+c_s\Delta t + (K+1)\delta t$, not at time $t=t^N+c_s\Delta t$.
Hence, one should, in principle, adjust the weight $b_s$ to keep the Gaussian
quadrature interpretation of the Runge--Kutta method, see, e.g., \cite{Lee2007}
for work in this direction. However, as will be shown in
section~\ref{sec:consist}, this additional consistency error will be negligible in the limit when $\varepsilon$ tends to $0$, which is the relevant limit in this paper.

In the numerical experiments, we will specifically use the projective Runge--Kutta
methods of orders 2 and 4 represented by the Butcher tableaux in
Figure~\ref{fig:butcher}.

\begin{figure}
\centering
\includegraphics[width=0.75\textwidth]{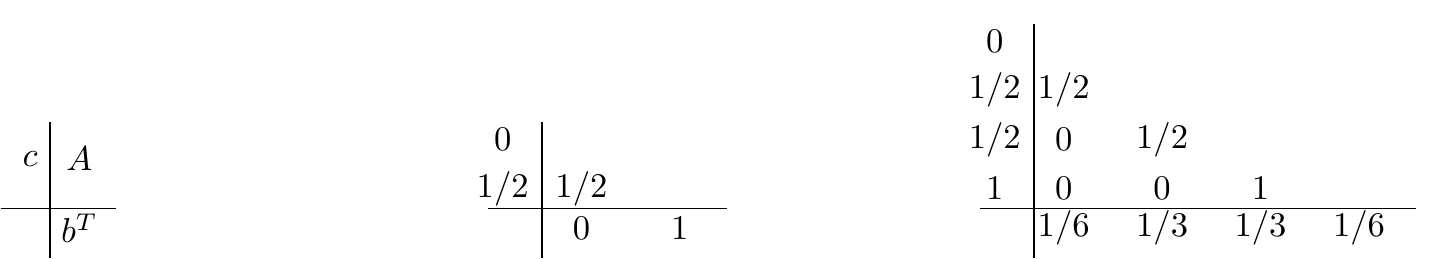}
	\caption {Butcher tableaux for Runge-Kutta methods. Left: general notation; middle: RK2 method (second order); right: RK4 method (fourth order).\label{fig:butcher}}
\end{figure}

\subsection{Stability of higher order projective integration}\label{sec:stab}

Let us now study the linear stability regions of the higher order
Runge--Kutta projective integration methods that were devised above. As is traditional, we
introduce to this end the Dahlquist test equation and its corresponding inner
integrator,
\begin{equation}\label{eq:test_eq}
	\dot{y}=\lambda y,  \qquad y^{k+1}=\tau(\lambda\delta t)y^k, \qquad \lambda < 0.
\end{equation}
As in \cite{gear:1091}, we call $\tau(\lambda\delta t)$ the \emph{amplification
factor} of the inner integrator. (For instance, if the inner integrator is
forward Euler, we have $\tau(\lambda\delta t)=1+\lambda\delta t$.) The inner
integrator is linearly stable if $\left|\tau\right|\leq 1$.
The analysis below will
reveal for which values of $\tau$ the projective integration method is also stable.
In section~\ref{sec:param}, this analysis will be combined with an analysis of the spectrum of the kinetic equation~\eqref{kin_eq} to determine the method parameters $\delta t$, $\Delta t$ and $K$ of the projective integration method.

A projective Runge--Kutta method applied to~\eqref{eq:test_eq} can be written as
\begin{equation}
	y^{N+1}=\sigma(\tau;\Delta t,\delta t, K)y^N,\label{eq:sigma}
\end{equation}
which is stable when $\left|\sigma(\tau;\Delta t,\delta t, K)\right|\leq 1$.
For projective forward Euler, we have 
\begin{equation}\label{e:stab_cond}
  \sigma^{PFE}(\tau;\Delta t,\delta t, K)=\left[\left(\dfrac{\Delta t-(K+1)\delta t}{\delta
          t}+1\right)\tau
      -\dfrac{\Delta t-(K+1)\delta t}{\delta t}\right]\tau^K.
\end{equation}

Given the kinetic equation~\eqref{kin_eq}, the goal in this paper is to take a projective time step $\Delta t =
\Oh(\Delta x^{\gamma+1})$, whereas $\delta t=\Oh(\varepsilon^{\gamma+1})$
necessarily to ensure stability of the inner brute-force forward Euler integration.  Since we are interested in the limit $\varepsilon\to 0$ for fixed $\Delta x$, we therefore look at the limiting stability regions as $\Delta t/\delta t\to \infty$.  In this regime, it is shown in \cite{gear:1091} that the values $\tau$ for which the condition \eqref{e:stab_cond} is satisfied lie in the union of two separated disks
$\mathcal{D}^{PFE}_1\cup\mathcal{D}^{PFE}_2$ where
\begin{equation}\label{eq:stab_pfe}
\mathcal{D}^{PFE}_1=\mathcal{D}\left(1-\dfrac{\delta t}{\Delta t},\dfrac{\delta
    t}{\Delta t}\right)\text{ and
}\mathcal{D}^{PFE}_2=\mathcal{D}\left(0,\left(\dfrac{\delta t}{\Delta
      t}\right)^{1/K}\right).
\end{equation}
The eigenvalues in $\mathcal{D}^{PFE}_2$ correspond to modes that are quickly
damped by the time-stepper, whereas the eigenvalues in $\mathcal{D}_1^{PFE}$
correspond to slowly decaying modes.  When the method parameters $\delta t$, $\Delta t$ and $K$ are suitably chosen, the projective integration method then
allows for accurate integration of the modes in $\mathcal{D}^{PFE}_1$ while
maintaining stability for the modes in $\mathcal{D}^{PFE}_2$.

We now show how the stability regions of higher order projective Runge--Kutta  schemes relate to those of projective forward Euler when $\delta t/\Delta t$ tends to $0$.
\begin{thm}[Stability of higher order projective Runge--Kutta methods]\label{thm:stab}
Assume the inner integrator is stable, i.e., $\left|\tau\right|\leq1$, and $\delta t$, $K$
and $\Delta t$ are chosen such that the projective forward Euler method is stable. Then, a projective Runge--Kutta method is also stable if it satisfies Assumptions~\ref{ass:RK} and the convexity condition
\begin{equation}\label{eq:rk_convex}
0 \le a_{s,l} \le c_s, \qquad \forall 1\le l \le s,\;\; 1\le s\le S.
\end{equation}
\end{thm}
Such a result is classical for regular Runge--Kutta methods \cite{Hairer1993}.  Here, however, we also provide the proof in the \emph{projective} Runge--Kutta case, to show that the above property holds both for the stability domain corresponding to slow eigenvalues and for the stability domain corresponding to quickly damped eigenvalues.

\begin{proof}
%
Let us first introduce, as in \cite{gear:1091},
$M=(\Delta t - (K+1)\delta t)/\delta t$ and, similarly, $M_s = (c_s\Delta
t-(K+1)\delta t)/\delta t$,
and remark that $M_s \leq c_s M \leq M$ is satisfied for all $s\in [1,S]$.
We can then rewrite the Runge--Kutta scheme \eqref{eq:rk1}-\eqref{eq:rk2}-\eqref{eq:rk3} for the test equation~\eqref{eq:test_eq}:
\begin{align}
	\begin{cases}
k_1 &=:\kappa_1(\tau) y^N =\dfrac{\tau^{K+1}-\tau^{K}}{\delta t}y^N \\
k_s &=:\kappa_s(\tau) y^N= \dfrac{\tau^{K+1}-\tau^{K}}{\delta
t}\left(\tau^{K+1}+(M_s\delta
t)\sum_{l=1}^{s-1}\dfrac{a_{s,l}}{c_s}\kappa_l\right)y^N, \qquad 2 \le s \le S,\label{eq:rk_linear}\\
 y^{N+1} &=: \sigma(\tau) y^N= \left(\tau^{K+1} + (M\delta t)\sum_{s=1}^{S}b_s
 \kappa_s\right)y^N ,
\end{cases}
\end{align}
where we have suppressed the dependence of $\kappa$ and $\sigma$ on $K$, $\delta t$ and $\Delta t$ but emphasized the dependence on $\tau$.

The proof then amounts to showing that the condition $|\sigma|\le
1\label{eq:left_to_prove}$
is satisfied as soon as the stability condition for the projective forward Euler scheme, i.e.,
\begin{equation}\label{eq:stab_fe}
\left|\left((M+1)\tau-M\right)\tau^{K}\right|\le 1,
\end{equation}
is satisfied. 
The proof is split up in three steps:
\begin{itemize}
	\item We first remark that if condition~\eqref{eq:stab_fe} is satisfied, this implies that
\begin{equation}\label{eq:stab_fe_all_Dt}
\left|\left((\alpha M+1)\tau-\alpha M\right)\tau^{K}\right|\le 1,
\end{equation}
for all $\alpha \in [0,1]$,  since $\alpha M/(\alpha M+1)\leq M/(M+1)$, so that $\mathcal{D}(M/(M+1))\subset \mathcal{D}(\alpha M/(\alpha M+1))$.
\item Next, we prove by induction that
\begin{equation}\label{eq:kappa_bound}
	\kappa_s \le \dfrac{\left|\tau^{K+1}-\tau^K\right|}{\delta t}, \qquad 1 \le s \le S.
\end{equation}
Clearly, this statement is true for $s=1$.
For $s>1$, we have
\begin{align}
	\kappa_s = \dfrac{\tau^{K+1}-\tau^K}{\delta t}\left(\tau^{K+1}+(M_s\delta
	t)\sum_{l=1}^{s-1}\dfrac{a_{s,l}}{c_s}\kappa_l\right).
\end{align}
Assume that for $n\in \{1,\hdots,s-1\}$, $s\geq 2$ :
\begin{equation}
	|\kappa_n| \leq \dfrac{\tau^{K+1}-\tau^K}{\delta t}. \label
	{ass:induction_kappa}
\end{equation}
We thus need to show that
\begin{equation}
	\left|\tau^{K+1}+(M_s\delta
	t)\sum_{l=1}^{s-1}\dfrac{a_{s,l}}{c_s}\kappa_l\right|\le 1.
\end{equation}
To this end, we write
\begin{align}
	\tau^{K+1}+(M_s\delta t)\sum_{l=1}^{s-1}\dfrac{a_{s,l}}{c_s}\kappa_l
	&=\tau^{K+1}+M_s\left(\tau^{K+1}-\tau^K\right)\sum_{l=1}^{s-1}\dfrac{a_{s,l}}{c_s}\dfrac{\kappa_l\delta
	t}{\tau^{K+1}-\tau^K}\nonumber\\
	&=\left((\alpha M_s+1)\tau-\alpha M_s\right)\tau^{K},
\end{align}
with
\[
\alpha =
\dfrac{M_s}{M}\sum_{l=1}^{s-1}\dfrac{a_{s,l}}{c_s}\dfrac{\kappa_l\delta
t}{\tau^{K+1}-\tau^K}.
\]
Using~\eqref{eq:rk_convex}, the induction hypothesis
\eqref{ass:induction_kappa} and the fact that $M_s\leq M$, we deduce that
$0\le \alpha\le 1$, from which, using \eqref{eq:stab_fe_all_Dt}, we
conclude~\eqref{eq:kappa_bound}.\\
\item Now we are ready to show that \eqref{eq:left_to_prove} holds, since the
  latest result is valid for $s=S$. Using the
same reasoning, we can rewrite $\sigma$ as:
\begin{align}
\sigma =\left((\beta M+1)\tau-\beta M\right)\tau^{k},\qquad \beta =
\sum_{s=1}^{S}b_s \kappa_s\dfrac{\delta t}{\tau^{K+1}-\tau^K},
\end{align}

from which, using \eqref{eq:kappa_bound}, $0\le \beta\le 1$ and assumptions~\ref{ass:RK},
we deduce~\eqref{eq:left_to_prove}.
\end{itemize}
\end{proof}

As for projective forward Euler, the stability region breaks up into two parts
when $\delta t/\Delta t$ tends to $0$.  By performing an asymptotic expansion of
$\sigma$ (see~\eqref{eq:sigma}) in terms of $\delta t/\Delta t$, we can obtain a
parameterisation of the boundary of both regions, defined by the set of values
$\tau$ for which $\left|\sigma(\tau;\Delta t,K,\delta t)\right|=1$. We have the
following result:

\begin{prop}\label{thm:stab_par}
In the limit when $\delta t/\Delta t$ tends to $0$, the stability region of a projective Runge--Kutta method consists of two regions $\mathcal{R}_1^{\mathrm{PRK}}\;\cup \;\mathcal{R}_2^{\mathrm{PRK}}$. The boundary of $\mathcal{R}_1^{\mathrm{PRK}}$ is given by an asymptotic expansion of the form
\begin{equation}
\tau(\theta) = 1+C_1(\theta)\left(\dfrac{\delta t}{\Delta t}\right)+C_2(\theta)
\left(\dfrac{\delta t}{\Delta t}\right)^2+\text{ h.o.t.},\qquad 0\le\theta\le 2\pi,
\end{equation}
%
whereras the boundary of $\mathcal{R}_2^{\mathrm{PRK}}$ can be expanded as
\begin{equation}
	\tau(\theta)=C_1'(\theta)\left(\dfrac{\delta t}{\Delta
			t}\right)^{1/K}+C_2'(\theta)\left(\dfrac{\delta t}{\Delta
			t}\right)^{2/K}+\text{ h.o.t.}, \qquad 0\le\theta\le 2\pi.
\end{equation}
\end{prop}
The proof, containing also the expressions for $C_{1,2}(\theta)$ and
$C_{1,2}'(\theta)$, is given in the Appendix, which also contains the 
expressions of the projective Runge--Kutta methods with Butcher tableaux in
Figure~\ref{fig:butcher}. An additional observation, which we will state here
without proof,
 is that in the limit when $\delta t/\Delta t$ tends to $0$, the
stability regions of lower order methods are contained within those of
higher-order methods, i.e., the stability regions satisfy
\[
\mathcal{R}_1^{PRK,p+1}\supseteq \mathcal{R}_{1}^{PRKp}\supseteq
\mathcal{D}^{PFE}_1 \quad \text{and}\quad \mathcal{R}_2^{PRKp+1}\supseteq
\mathcal{R}_{2}^{PRKp}\supseteq \mathcal{D}^{PFE}_2, \qquad \forall p\ge 1,
\]
in which the integer $p$ indicates the order of the method.

We illustrate the
shape of these stability domains for the classical second-order and fourth order
Runge--Kutta method whose tableaux are given in figure~\ref{fig:butcher}.
The stability regions are shown in figure~\ref{fig:stab_dt}.  The figure illustrates theorem~\ref{thm:stab}, and additionally shows that the stability regions scale with $\Delta t/\delta t$ in the same way as for the projective forward Euler method.  The shape of the stability regions, however, depends on the method used.  It can be checked that the region $\mathcal{R}^{PRK}_1$ converges to the stability domain of the corresponding classical Runge--Kutta method when $\delta t/\Delta t$ tends to $0$.

\begin{figure}[h]
	\centering
 	\includegraphics[width=\textwidth]{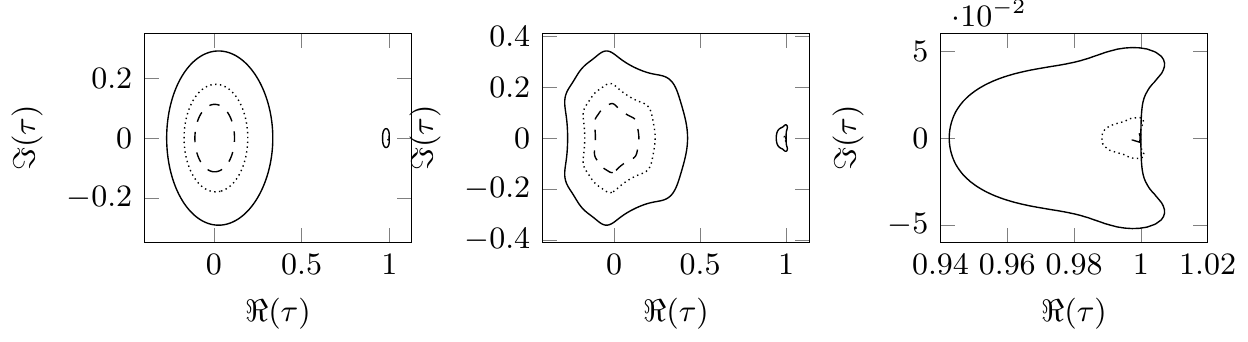}
	\caption{The two leftmost pictures show (respectively) the stability regions for the PRK2
	method and the PRK4 method, while the picture on the right shows a zoom on the
	region of the PRK4 method near $1$. Parameters: $\Delta t=\num{1e-3},K=3$ and
	$\delta t=\num{1e-6}$ (dashed), $\delta t=\num{4e-6}$ (dotted) and $\delta t=\num{1.6e-5}$.\label{fig:stab_dt}}
\end{figure}

The main conclusion of the above analysis is that, whereas the stability regions of higher order projective Runge--Kutta methods differ from those of projective forward Euler in their precise shape, their qualitative dependence on the parameters of projective integration ($\delta t$, $K$ and $\Delta t$) is the same, and method parameters that are suitable for projective forward Euler, will also be suitable for the higher order projective Runge--Kutta method.

\section{Stability analysis}\label{sec:spectrum}

We are now ready to study the stability of the projective integration schemes for the kinetic equation~\eqref{kin_eq}. After introducing some notation in section~\ref{sec:notation}, we compute bounds on the spectrum of the inner integrator \eqref{eq_disc_scheme} with a linear Maxwellian~\eqref{eq:maxw-proto} with $A(u)=u$ in section~\ref{sec:spec_inner}. Subsequently, we look into suitable parameter choices for the projective integration schemes (section~\ref{sec:param}).

\subsection{Notation and assumptions}\label{sec:notation}

We first rewrite the semi-discretized kinetic equation~\eqref{eq:semidiscrete} in the (spatial) Fourier domain,
\begin{equation}\label{eq:fourier_semidiscr}
\partial_t \hat{F}(\zeta) = B\; \hat{F}(\zeta), \qquad \text{with}\qquad
B = \dfrac{1}{\varepsilon^{\gamma+1}}\left(\varepsilon D  +MP-\Id \right),
\end{equation}
with $\hat{F}\in \R^J$, the matrices $B$, $M$,
$P \in \R^{J\times J}$, and  $\Id$ the identity matrix of dimension $J$.
In~\eqref{eq:fourier_semidiscr}, the matrix $D$ represents minus the
(diagonal) Fourier matrix of the spatial discretisation chosen for the convection part, $P$ is the rank $1$ Fourier matrix of the averaging of $f$ over all
velocities,
\[
P:= ee^T, \qquad e=\dfrac{1}{\sqrt{J}}(1,\ldots,1)^T \in \R^J,
\]
and the invertible matrix $M$ represents the Fourier transform of the
Maxwellian, $ M = \Id + \varepsilon^\gamma V^{-1}$, with $V$ the diagonal matrix
with elements $v_j$ given in \eqref{eq:defV}. For the spatial discretisations in equations~\eqref{eq:derivx-par} and~\eqref{eq:derivx-hyp}, the matrix $D$ is 
\begin{equation}
\begin{cases}
	D= -\imag \dfrac{8\sin(\zeta)-\sin(2\zeta)}{6\Delta x}V & \text{
(parabolic),}\\[0.5cm]
	D =	-\left(\dfrac{3-4\cos(\zeta)+\cos(2\zeta)}{6\Delta x}\right)V
	-\imag \left(\dfrac{8\sin(\zeta)-\sin(2\zeta)}{6\Delta x}\right)V&
	\text{(hyperbolic).}
	\end{cases}
\end{equation}
From now on, we write $D_j = \alpha_j + \imath \beta_j$ for $j\in\{1,\hdots,J\}$. Thus, we have 
\begin{equation}\label{eq:alpha-upwind}
	\alpha_j=-\dfrac{|v_j|}{6\Delta x}{\left(3-4\cos(\zeta)+\cos(2\zeta)\right)}
	\qquad \beta_j=-\dfrac{v_j}{6\Delta x}{\left(8\sin(\zeta)-\sin(2\zeta)\right)},
\end{equation}
for the third order upwind scheme, whereas 
\begin{equation}\label{eq:alpha-central}
\alpha_j=0 \qquad \beta_j=-v_j\dfrac{8\sin(\zeta)-\sin( 2\zeta)}{6\Delta x}.
\end{equation}
for the fourth order central scheme. We also define
\begin{equation}
	\tilde{e}=Me= (\Id+\varepsilon^\gamma V^{-1})e \label{eq:etilde},
\end{equation}
from which we obtain $MP=\tilde{e}e^T$. We write the Fourier transform of~\eqref{eq_disc_scheme} as
 \begin{equation}\label{eq:fourier_fe}
\hat{F}^{k+1} = S_{\delta t}\;\hat{F}^{k} =
\left(\Id+\delta t B \right)\hat{F}^k =
\left(1-\dfrac{\delta t}{\varepsilon^{\gamma+1}}\right)\Id +\dfrac{\delta
t}{\varepsilon^{\gamma+1}}A,
\end{equation} where  $A$ is defined as $MP+\varepsilon D$. It is clear that the
amplification factors ${\mathbf{\tau}=(\tau_1,\ldots,\tau_J)}$ of the forward Euler scheme (which are
the eigenvalues of $S_{\delta t}$)  and the eigenvalues
$\lambda=(\lambda_1,\ldots, \lambda_J)$ of the matrix $A$ are related
via
\[
\tau_j = \left(1-\dfrac{\delta t}{\varepsilon^{\gamma+1}}\right) + \dfrac{\delta
t}{\varepsilon^{\gamma+1}} \lambda_j, \qquad 1 \le j \le J.
\]

To locate the spectrum, we
assume the velocity space is symmetric (see \eqref{eq:defV}), 
\begin{equation}\label{eq:ass2}
v_{J-j}=-v_j  \qquad 1\le j \le J/2,
\end{equation}
so that
\begin{equation}
  D_{J-j}=\overline{D}_j  \qquad 1\le j \le J/2.
\end{equation}
\subsection{Spectrum of the inner integrators}\label{sec:spec_inner}

We have the following result for the spectrum of $A=MP+\varepsilon D$.

\begin{thm}
Under the assumptions in section~\ref{sec:notation}, the spectrum of the matrix
$A=MP+\varepsilon D$ satisfies
    \begin{equation*}
      \mbox{\emph{Sp($A$)}}\subset \left(\mathcal{D}\left(0,\varepsilon C\max_{j\in
            \J^+}(|\alpha_j|+|\beta_j|)\right)\right) \cup \{\lambda(\varepsilon)\}
    \end{equation*}
    where the constant $C$ depends on the parameters $(\alpha_j)_{j=1}^J$ and
    $(\beta_j)_{j=1}^J$ of the spatial discretisation scheme and the chosen
    velocities $(v_j)_{j=1}^J$. The dominant eigenvalue $\lambda(\eps)$ is
    simple and can be expanded as
\begin{eqnarray}
\Re(\lambda(\varepsilon))&=&1+\eps\langle
\alpha\rangle+\eps^2\left(\langle(\langle\alpha\rangle -\alpha)^2
\rangle-\langle\beta^2 \rangle+ \delta_\gamma \Big\langle\dfrac{\beta}{v}
\Big\rangle^2\right)+\oh(\eps^2)\\
\Im(\lambda(\varepsilon)) &=&  \eps \delta_\gamma \Big\langle \dfrac{\beta}{v}\Big\rangle +\eps^2
\left(\delta_\gamma\Bigg\langle \left(\Big\langle \dfrac{\beta}{v}\Big\rangle
-\beta\right)\left(\langle\alpha \rangle-\alpha\right)
\Bigg\rangle+\delta_{\gamma-1}\Big\langle\dfrac{\beta}{v} \Big\rangle\right)+\oh(\eps^2),
\end{eqnarray}
where we used $\delta_\gamma$ in the sense of the classical Kronecker delta symbol,
where $\delta_\gamma=1$ if $\gamma=0$ and zero otherwise.
\label{thm_spec}
\end{thm}

The proof of theorem \ref{thm_spec} has the same structure as the proof in
\cite{lafitte2012}. However, due to the presence of the Maxwellian $\Maxw$, each
of the intermediate steps becomes more involved. We split up these steps in
several lemmas.

\begin{lem}
 The rank-one matrix $MP$ is a projection matrix.
\end{lem}
\begin{proof}
We need to show that $(MP)^2=MP$.
Using the definitions introduced above, we get
\begin{eqnarray*}
	(MP)^2 &=& (\etildeexpr)(ee^T)^2+
		\varepsilon^\gamma(\etildeexpr)ee^TV^{-1}ee^T\\
		&=& (\etildeexpr)ee^T=MP,
\end{eqnarray*}
where, in the last line, we have used (i) the fact that $e^TV^{-1}=0$ to eliminate the second term (since the velocity space is assumed to
be odd), and (ii) $e^Te=1$, from which we obtain $(ee^T)^2=ee^T$.
\end{proof}

The following corollary is an immediate consequence.

\begin{cor}The matrix $MP$  has one eigenvalue $\lambda_1=1$ and all other
  eigenvalues vanish, i. e.  $\lambda_j=0$, $2\le j \le J$.
\end{cor}

\begin{lem} Consider the matrix $A=MP+\tilde{D}$, and assume
	\begin{equation}\label{eq:ass1}
		\tilde{D}=\diag(\tilde{D}_1,\ldots,\tilde{D}_J),
		\text{ where } \tilde{D}_{j}=\tilde{D}_{j'} \text{ implies } j=j'.
			\end{equation}
		 Then, the eigenspaces of $A$ are of dimension $1$ and no $\tilde{D}_j, 1\le j
		 \le J$ is an eigenvalue of  $A$.
\end{lem}
\begin{proof}
Let $(\lambda,W)$ be an eigenvalue and an associated
eigenvector of $A$. This implies
\begin{align}
	\left(MP+\tilde{D}\right)W=\lambda W& 
	\langle W\rangle \tilde{e}+ \tilde{D}W=\lambda W, \quad \text{with }
\langle W\rangle = e^T W. \label{eq:evv}
\end{align}

Assume now $\langle W\rangle=0$, from which we infer that
$\tilde{D}W=\lambda W$. Since $\langle
W\rangle=0$ with $W\neq 0$, there exists at least two indices $j_1$ and $j_2$
such that $W_{j_1},W{j_2}\neq 0$. However, this implies that
$\lambda=D_{j_1}=D_{j_2}$, which violates assumption~\eqref{eq:ass1}.\

So necessarily $\langle W\rangle \neq 0$. Then \eqref{eq:evv} implies $
W=\left(\lambda \Id - \tilde{D}\right)^{-1}\langle W\rangle \tilde{e}$
that is, all the eigenspaces are of dimension $1$ and no $D_j$ can be an eigenvalue of $A$.
\end{proof}

Let us, from now on, choose $\tilde{D}=\varepsilon D$, and investigate the
matrix $A=MP+\varepsilon D$.

\begin{lem}
Introducing $Q(\lambda):=\prod_{j=1}^{J}(\varepsilon D_j-\lambda)$, 
the characteristic polynomial $\chi_A(\lambda)$ of
$A=MP+\varepsilon D$ can be written as
\begin{equation} \chi_A(\lambda)=Q(\lambda)\left(1-\dfrac{1}{J}\displaystyle\sum_{j=1}^{J}\dfrac{1+\varepsilon^\gamma/v_j}{\lambda-\varepsilon
  D_j}\right) \qquad .
\end{equation}
\end{lem}
\begin{proof}
We start by writing
\begin{equation*}
\chi_A(\lambda)=\begin{vmatrix}
	\tilde{e}_1/\sqrt{J}+\varepsilon
	D_1 -\lambda &
	\tilde{e}_1/\sqrt{J}&
	\hdots&
	\tilde{e}_1\sqrt{J}\\[0.5cm]
		\tilde{e}_2/\sqrt{J} &
	\tilde{e}_2/\sqrt{J}+\varepsilon
	D_2 - \lambda&
	\hdots&
	\tilde{e}_2/\sqrt{J}\\
	\vdots & \vdots & \ddots &   \vdots \\
	\tilde{e}_J/\sqrt{J} &
	\tilde{e}_J/\sqrt{J} &\hdots &
	\tilde{e}_J/\sqrt{J}
	+\varepsilon D_J -\lambda
\end{vmatrix},
\end{equation*}
with $\tilde{e}_j=(1/\sqrt{J})\left(1+\varepsilon^\gamma/v_j\right)$.
This is the determinant of an arrow matrix
\begin{equation*}
  a=
  \begin{pmatrix}
    d_1 & r_2 & & \hdots &   r_{J}\\
    c_2 & d_2 & 0 & & \\
    c_3 & 0 & d_3 & 0 & \\
\vdots & & & \ddots & \\
c_{J} & &  & & d_{J}
  \end{pmatrix},
\end{equation*}
the determinant of which is
\begin{equation*}
  \det(a)=\prod_{j=1}^{J}d_j-\displaystyle\sum_{j=2}^{J}\left(c_j r_j \prod_{j'=2,j'\neq j}^{J}d_{j'}\right).
\end{equation*}
 So, after identifying $d_1=\tilde{e}_1/\sqrt{J}+\varepsilon D_1-\lambda$ and
 $d_j=\varepsilon D_j-\lambda$,  $c_j=\tilde{e}_j/\sqrt{J}$ and
 $r_j=-(\varepsilon D_1-\lambda)$, for $2\leq j\leq J$ some elementary manipulations yield $\chi_A(\lambda)$
\begin{align}
\chi_A(\lambda)
&=Q(\lambda)\left(1-\dfrac{1}{J}\displaystyle\sum_{j=1}^{J}\dfrac{1+\varepsilon^\gamma/v_j}{\lambda-\varepsilon
  D_j}\right),
\end{align}
where we made use of~\eqref{eq:etilde}. This concludes the proof.
\end{proof}

To prove theorem \ref{thm_spec}, we will consider the characteristic polynomial
$\chi_A(\lambda)$ to be a perturbation of the characteristic polynomial that was studied in \cite{lafitte2012}. We recall the following theorem from \cite{lafitte2012} in the notation of the present paper.
\begin{thm}[Proposition~4.1 in \cite{lafitte2012}]\label{thm:lafitte}    Consider the
matrix $A_0=P+\varepsilon D$, assuming~\eqref{eq:ass1} and~\eqref{eq:ass2}.
Then, the corresponding characteristic polynomial is
\begin{equation}
 \chi_{A_0}(\lambda)=Q(\lambda)\left(1-\dfrac{1}{J}\displaystyle\sum_{j=1}^{J}\dfrac{1}{\lambda-\varepsilon
	  D_j}\right)
\end{equation}
and its eigenvalues satisfy
    \begin{equation*}
      \mbox{\emph{Sp($A$)}}\subset \left(\mathcal{D}\left(0,\dfrac{\varepsilon}{J}\max_{j\in
            \J^+}(|\alpha_j|+|\beta_j|)\right)\right) \cup \{\lambda(\varepsilon)\}
    \end{equation*}
    where the real eigenvalue $\lambda(\eps)$ is simple and can
    be expanded as
    \begin{equation*}
      \lambda(\varepsilon)=1-\varepsilon \dfrac{\langle\alpha\rangle}{J}
        -\dfrac{\varepsilon^2}{J^2}\langle(\alpha-\langle\alpha\rangle)^2+\beta^2\rangle+o(\varepsilon^2).
    \end{equation*}
\end{thm}

Since we know how to localize the roots of $\chi_{A_0}$, we can use Rouch\'e's
theorem \cite{Ulrich2008} to bound the eigenvalues of $\chi_A$.
\begin{prop}[Rouch\'e's theorem] If there exists a closed simple
contour $\zeta$ in $\mathbb{C}$ encircling a compact $\mathcal{C}$, such that
\begin{equation}
\forall \lambda \in \zeta, \chi_{A_0}(\lambda)\neq 0\quad\text{and}\quad |\chi_A(\lambda)-\chi_{A_0}(\lambda)|<|\chi_{A_0}(\lambda)|,\label{rouche}
\end{equation}
then $\chi_A$ and $\chi_{A_0}$ have exactly the same number of roots in $\mathcal{C}$.
\end{prop}
Everything is now in place to prove theorem~\ref{thm_spec}.

\begin{proof}[Proof of theorem \ref{thm_spec}]
	The proof consists of two steps.  First, we will construct, using Rouch\'e's theorem, contours in which the eigenvalues of $\chi_{A}$ are known to be localized.  In a second step, we will provide an asymptotic expansion for the dominant eigenvalue.

\paragraph{Step (i): Localization of eigenvalues}
We start by writing
\begin{equation}
\chi_A(\lambda)=\chi_{A_0}(\lambda)+\dfrac{\varepsilon^\gamma}{J}\sum_{j=1}^{J}\dfrac{R_j(\lambda)}{v_j},
\end{equation}
and aim at applying Rouch\'e's theorem.
We thus study the rational function
\begin{equation}
\mathcal{F}:\lambda\mapsto
\dfrac{\chi_A(\lambda)-\chi_{A_0}(\lambda)}{\chi_{A_0}(\lambda)}=-\dfrac{\dfrac{\varepsilon^\gamma}{J}
\displaystyle\sum_{j=1}^{J}\dfrac{1}{v_j}\dfrac{1}{\lambda-\varepsilon
 D_j}}{1-\dfrac{1}{J}\displaystyle\sum_{j=1}^{J}\dfrac{1}{\lambda-\varepsilon
 D_j}}
\end{equation}
and look for contours that contain the eigenvalues of $\chi_{A_0}$ and for which $\left|\mathcal{F}(\zeta)\right|<1$.
\begin{itemize}
\item Let us first consider the dominant eigenvalue by enclosing the dominant eigenvalues of $\chi_{A_0}$ in a circle around $\lambda=1$. To this end, we search a value of $r>0$ such that \eqref{rouche} is satisfied on $\zeta=\{1+\eps^{\gamma+1}
re^{\imag\theta},\,\theta\in[0,2\pi)\}$.
Performing a Taylor expansion of $1/(1+\eps^{\gamma+1}r e^{\imag\theta} -\eps D_j )$
in terms of $\eps$ yields
\begin{eqnarray*} 
\mathcal{F}(1+\varepsilon^{\gamma+1}r e^{\imag\theta}) &=& \dfrac{
-\dfrac{\varepsilon^\gamma}{J} \displaystyle\sum_{j=1}^J
\dfrac{1}{v_j}\left(\varepsilon D_j -\varepsilon^{\gamma+1} r e^{\imag\theta}\right)+
\Oh(\varepsilon^2) }{1-\dfrac{1}{J}\sum _{j=1}^J 1-\varepsilon^{\gamma+1}r
e^{\imag\theta}+\varepsilon D_j +\Oh(\varepsilon^2)}\\\\
&=& \dfrac{ -\dfrac{\varepsilon^\gamma}{J} \sum_{j=1}^J \dfrac{D_j}{v_j}
+\Oh(\varepsilon)}{\varepsilon^\gamma r e^{\imag\theta}-\dfrac{1}{J}\sum_{j=1}^J D_j
+\Oh(\varepsilon)}
\end{eqnarray*}
where we have used the fact that $\dfrac{1}{J}\sum_j v_j=0$. 
When choosing $r$ such that
$r>\dfrac{1}{J}\left|\sum_{j=1}^{J}D_j\right|+\dfrac{2}{J}\left|\sum_{j=1}^{J}\dfrac{D_j}{v_j}\right|$, we ensure that
  $|\mathcal{F}(\lambda)|<1/2+\Oh(\eps)$,
from which one can conclude that $\chi_A$ and $\chi_{A_0}$ have the same number of zeroes, that is, $1$,
around $\lambda=1$ in a neighbourhood of size $\eps^{\gamma+1}$.

\item Let us now consider the $J-1$ remaining eigenvalues by considering the region around $\lambda=0$.
Again, we will make use of Rouch\'e's theorem: let us find $r>0$ such that \eqref{rouche} is satisfied on
$\zeta=\lbrace{\varepsilon r e^{\imag\theta},\theta\in[0,2\pi)\rbrace}$.
We thus study
\begin{equation}
\mathcal{F}(\varepsilon  r e^{\imag\theta})=-(\Fact) \dfrac{\phi(r)}{\varepsilon-\psi(r)}
\end{equation}
with
\begin{equation}
\psi:r\mapsto \dfrac{1}{J}\displaystyle\sum_{j=1}^{J}\dfrac{1}{r
e^{\imag\theta}-D_j},\qquad \phi(r):r\mapsto
\dfrac{1}{J}\displaystyle\sum_{j=1}^{J}\dfrac{1}{v_j}\dfrac{1}{r
e^{\imag\theta}-D_j}.
\end{equation}
Performing a Taylor expansion  of $\phi(r)$ and $\psi(r)$ in $1/r$ yields:
\begin{equation}
\psi(r)=\dfrac{e^{-\imag\theta}}{r}+{\Oh}\left(\dfrac{1}{r^2}\right),\quad
\phi(r)=\dfrac{e^{-2\imag\theta}}{Jr^2}\displaystyle\sum_{j=1}^{J}
\dfrac{D_j}{v_j}+{\Oh}\left(\dfrac{1}{r^3}\right),
\end{equation}
so
\begin{equation}
\mathcal{F}(\varepsilon r
e^{\imag\theta})=-(\Fact)\dfrac{e^{-\imag 2\theta}}{r^2}\dfrac{\dfrac{1}{J}\displaystyle\sum_{j=1}^{J}\dfrac{D_j}{v_j}+\Oh\left(\dfrac{1}{r}\right)}{\varepsilon-\dfrac{e^{-\imag\theta}}{r}+\Oh\left(\dfrac{1}{r^2}\right)}=(\Fact)\dfrac{e^{-\imag\theta}}{r}\dfrac{1}{J}\displaystyle
\sum_{j=1}^{J}\dfrac{D_j}{v_j}(1+\varepsilon r e^{\imag\theta}+\Oh(\varepsilon^2)).
\end{equation}

Choose $r>2\max \left(\dfrac{\varepsilon^\gamma}{J}\Bigg|\displaystyle
\sum_{j=1}^{J}\dfrac{D_j}{v_j}\Bigg|,1\right)$  to ensure that the main term
in the Taylor expansion is in modulus less than $1/2$. Thus we can conclude that
there are exactly $J-1$ eigenvalues in a neighbourhood of $\lambda=0$ of size $\varepsilon$.

\end{itemize}

\paragraph{Step (ii): Asymptotic expansion of dominant eigenvalue}
To obtain an asymptotic expansion of the dominant eigenvalue, we first define
\begin{equation}
 S_j := \dfrac{1+ \varepsilon^\gamma/v_j}{\lambda -\varepsilon D_j} +
 \dfrac{1 - \varepsilon^\gamma/v_j}{\lambda - \varepsilon\bar{D}_j},\quad
 \text{as well as}\quad \Sigma(\lambda(\eps)):=\frac{1}{J}\sum_{j=1}^{J/2}S_j.
\end{equation}

Now, given that $\lambda(\eps)$ (close to $1$) is a root of the characteristic polynomial $\chi_A(\lambda)$, we have $\Sigma(\lambda(\eps))=1$.
We therefore perform a Taylor expansion of $\Sigma(\lambda(\eps)$ around
$\varepsilon=0$.  We split $\lambda$ in its real and imaginary part:
$\lambda(\eps)=x(\eps)+\imath\; y(\eps)$ and proceed by requiring (up to second
order) \begin{equation} \Sigma(\lambda(\epsilon))\equiv\Sigma(\varepsilon)=\Sigma(0)+\varepsilon\Sigma'(0)+\frac{\varepsilon^2}{2}\Sigma''(0)=1.
\end{equation}
Matching, for all powers of $\eps$ the real and imaginary parts of the left and
right hand side, yields the conditions: $\Re(\Sigma(0))=1, \Im(\Sigma(0))=0$ and
$\Re(\Sigma^{(j)}(0))=\Im(\Sigma^{(j)}(0))=0\quad \forall j\geq 1$.
From these conditions, asymptotic expansions of $x(\eps)$ and $y(\eps)$ in terms
of $\eps$ can be derived.  Let us write
\begin{equation*}
	x(\eps)=x_0 + \eps x_1 + \eps^2 x_2 + \Oh(\eps^3), \qquad y(\eps)=y_0 + \eps y_1 + \eps^2 y_2 + \Oh(\eps^3)
	\end{equation*}
Then, we get, for the zeroth order term,
 \begin{equation}
\dfrac{2}{J} \displaystyle\sum_{j=1}^{J/2} \dfrac{x_0}{x_0^2+y_0^2}=1 \qquad
-\dfrac{2}{J}\displaystyle\sum_{j=1}^{J/2}
\dfrac{y_0\left(1+\Fact/v_j\right)}{x_0^2+y_0^2}=0,
 \end{equation}
 which implies that $x_0=1$ and $y_0=0$.  (This is consistent with the derivation based on Rouch\'e's theorem above.)

Next, we determine the terms of order $\eps$,
\begin{equation}
\Re(\Sigma)(0)=\displaystyle-\dfrac{2}{J}\sum_{j=1}^{J/2}x_1-\alpha_j \qquad
\Im(\Sigma)(0)=\dfrac{2}{J}\displaystyle\sum_{j=1}^{J/2}
\left(-y_1+\delta_\gamma\dfrac{\beta_j}{v_j}\right),
 \end{equation}
from which we conclude that  $x_1=\langle \alpha\rangle$ and
$y_1=\delta_\gamma\langle \beta/v\rangle$. Finally, for the second order terms, we
find $$
\begin{cases}
\displaystyle\sum_{j=1}^{J/2} -2x_2+4(\langle
\alpha\rangle-\alpha_j)^2-4(y_1^2+\beta_j^2)+8\delta_\gamma
y_1\dfrac{\beta_j}{v_j}=0\\
\displaystyle\sum_{j=1}^{J/2}-2y_2+4\delta_\gamma\left(\delta_\gamma\Big\langle\dfrac{\beta}{v}\Big\rangle-\beta_j\right)(\langle
\alpha\rangle -\alpha_j)+\delta_{\gamma-1} \dfrac{\beta_j}{v_j}=0,
\end{cases}
$$
from which we  conclude that
\begin{eqnarray}
x_2&=&2\left\langle
	\left(
		\left\langle\alpha\right\rangle-\alpha
	\right)^2
	-\beta^2
	+\delta_\gamma\left\langle\dfrac{\beta}{v}\right\rangle^2
	\right\rangle\\
 y_2&=&2\delta_\gamma\Big\langle \left(\Big\langle \dfrac{\beta}{v}
 \Big\rangle-\beta\right)(\langle \alpha\rangle -\alpha)
 \Big\rangle+2\delta_{\gamma-1}\Big\langle \dfrac{\beta}{v}\Big\rangle
 \end{eqnarray}
 Combining all terms concludes the proof.
\end{proof}

As an immediate consequence of the above theorem, we have:
\begin{cor}
Under the assumptions in section~\ref{sec:notation}, the spectrum of the matrix
$\mathcal{S}_{\delta t}$, corresponding to the Fourier-transformed forward Euler
time-stepper for the kinetic equation~\eqref{kin_eq} (defined in equation~\eqref{eq:fourier_fe})  is
located in two clusters
    \begin{equation*}
      \emph{Sp}(S_{\delta t})\subset \mathcal{D}_1 \cup \mathcal{D}_2
\end{equation*}
with
\begin{equation*}\mathcal{D}_1=\{\tau_{\delta t}\}\text{  and  }   \mathcal{D}_2=\mathcal{D}\left(1-\dfrac{\delta
		t}{\varepsilon^{\gamma+1}},\dfrac{\delta
		t}{J\varepsilon^\gamma}\max_{j\in\mathcal{J}}(|\alpha_j|+\beta_j|)\right)
    \end{equation*}
The dominant eigenvalue $\tau_{\delta t}$ is simple
and can be expanded as
\begin{eqnarray}
	\tau_{\delta t}&=& \left(1-\dfrac{\delta
	t}{\varepsilon^{\gamma+1}}\right)+\dfrac{\delta t}{\varepsilon^{\gamma+1}}\left(1+\eps\langle \alpha\rangle+\eps^2\left(\langle (\langle
\alpha\rangle -\alpha)^2 \rangle-\langle\beta^2 \rangle+\delta_\gamma
\Big\langle\dfrac{\beta}{v} \Big\rangle^2\right)\right)\\
&+\imag& \dfrac{\delta t}{\varepsilon^{\gamma+1}}\left(\eps \delta_\gamma \Big\langle
\dfrac{\beta}{v}\Big\rangle +\eps^2 \left(\delta_\gamma\Bigg\langle \left(\Big\langle \dfrac{\beta}{v}\Big\rangle
-\beta\right)\left(\langle\alpha \rangle-\alpha\right)
\Bigg\rangle+\delta_{\gamma-1}\Big\langle\dfrac{\beta}{v} \Big\rangle\right)\right)+\oh(\varepsilon^{1-\gamma}).
\end{eqnarray}
\label{cor:eig}
\end{cor}

These spectra are illustrated in figure \ref{fig:spectra_fe}, where we have
plotted the spectra of the amplification factor of the time-stepper
$\mathcal{S}_{\delta t}$ in the spatial domain (see equation~\eqref{scheme_time-stepper}) for several choices of $\delta t$ and for both the
parabolic and the hyperbolic scaling.

\begin{figure}
\includegraphics[width=\textwidth]{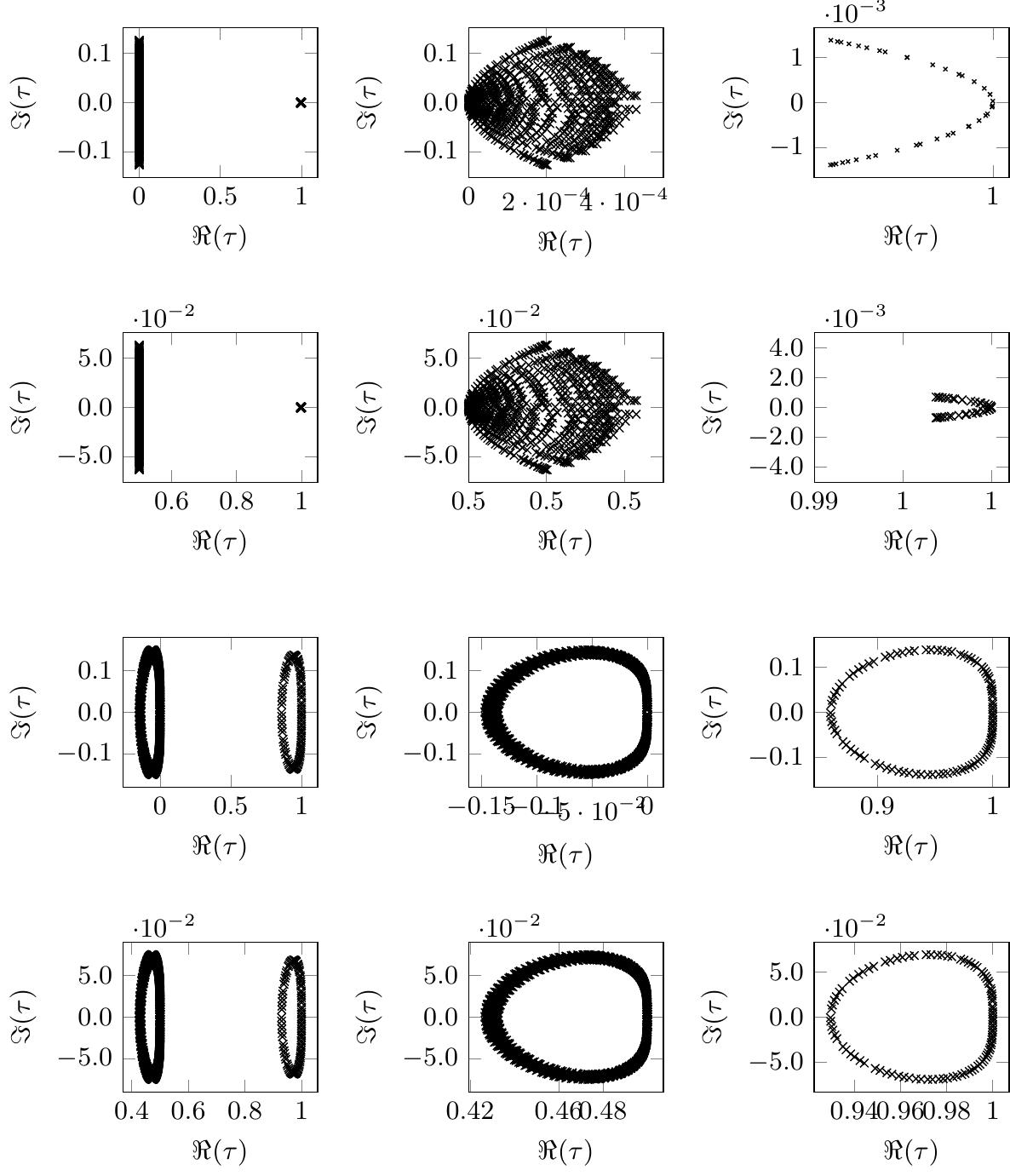}
		\caption{Spectrum of the time-stepper $\mathcal{S}_{\delta t}$. The first two
		rows correspond to a parabolic scaling ($\gamma=1$); the last two rows to a
	hyperbolic scaling ($\gamma=0$). On each row: Left: a global view of the spectrum; Middle
	and right: zoom to each of the eigenvalue clusters around $0$ and $1$. Parameter values are
	$K=3,\Delta t=\num{1e-2},\Delta x=0.1, J=20$ and $\delta
	t=\varepsilon^{\gamma+1}$ (first row) and $\delta
	t=\num{5e-1}\varepsilon^{\gamma+1}$ (second row)\label{spec_1}.\label{fig:spectra_fe} }
\end{figure}

\subsection{Parameter choices for projective integration}\label{sec:param}

Based on the expressions for the spectrum of the inner time-stepper \eqref{scheme_time-stepper} in corollary~\ref{cor:eig} and the stability regions of the projective Runge--Kutta methods in theorem~\ref{thm:stab}, we can determine parameter values $\delta t$, $\Delta t$ and $K$ for which the projective Runge--Kutta methods are stable.
We first observe from corollary \ref{cor:eig} that we need 
  $$\delta t =\varepsilon^{\gamma+1}$$ to center the fast eigenvalues of the inner time-stepper (corresponding to the region $\mathcal{D}_2$) around
  the origin to contain them in the stability region $\mathcal{R}_2$ (see theorem~\ref{thm:stab}).
\begin{rem}[Spatial mesh width] As observed in \cite{lafitte2012}, we remark that this choice induces a restriction on the spatial mesh width to ensure stability of the inner integrator. Specifically, we require
	\[
	\dfrac{\varepsilon}{J}\max_j\left(\left|\alpha_j\right| +\left|\beta_j\right| \right)\le 1,
	\]
	from which, using~\eqref{eq:alpha-upwind} or~\eqref{eq:alpha-central},
        it follows that $\Delta x \ge C v_J\varepsilon$.  However, since we
        consider the limit when $\varepsilon$ tends to $0$ for fixed $\Delta x$,
        as we are interested in Asymptotic Preserving schemes, this is not a problematic restriction.
\end{rem}

Next, we have to determine $\Delta t$ such that the slow
  eigenvalues are captured in $\mathcal{R}_1$ and choose $K$ in such a way that the stability region $\mathcal{R}_2$ is large enough to contain all fast eigenvalues.  We have the following conditions:

\begin{thm}[Stability of projective Runge--Kutta methods]
	When using an inner integrator~\eqref{scheme_time-stepper} for the kinetic equation~\eqref{kin_eq} with time step $\delta t=\varepsilon^{\gamma+1}$, a projective Runge--Kutta method~\eqref{eq:rk1}-\eqref{eq:rk3} is stable if
 the macroscopic time step $\Delta t$ satisfies
	  \begin{equation}\label{eq:Dt}
	   	\Delta t\leq \varepsilon^\gamma b,
	  \end{equation}
and the integer $K$ that determines the number $K+1$ of inner steps satisfies	   \begin{align}
	  	K&\geq
	\dfrac{1}{1+\dfrac{\log(\max_j(|\alpha_j|+|\beta_j|)/J)}{\log(\varepsilon)}}-\dfrac{\log(b)}{\log(\varepsilon)+\log(\max_j(|\alpha_j|+|\beta_j|)/J)}\label{eq:K},\\
	b&=\min\left(\dfrac{2}{-\langle \alpha\rangle
-\varepsilon(\Big\langle(\malph -\alpha)^2-\beta^2
\Big\rangle-\delta_\gamma\mbet^2)}, \dfrac{1}{\left|\delta_\gamma\mbet
	+\eps\left[\delta_\gamma\left(\mbet-\beta\right)(\malph-\alpha)+\delta_{\gamma-1}\mbet\right]\right|}\right). \label{eq:b}
	  \end{align}
	\label{thm:stabcond}
\end{thm}

Before proceeding to the proof, we make a few observations on the macroscopic
time step $\Delta t$. At first, consider the hyperbolic scaling ($\gamma=0$). In
this regime, the macroscopic time step $\Delta t$ is seen to be independent of
$\varepsilon$ when $\varepsilon$ tends to $0$.  Moreover, since the coefficients
$\alpha$ and $\beta$ depend on $1/\Delta x$, the inequality in condition~\eqref{eq:Dt} will
result in a CFL-type condition of the form $\Delta t \le C\Delta x$.  Now
consider the parabolic scaling ($\gamma=1$). In that case, the first term in
equation~\eqref{eq:Dt} can only be bounded independently of $\varepsilon$ if
$\langle \alpha \rangle=0$, i.e., by a central scheme. (This is consistent with
the observation in \cite{lafitte2012}.)  The second term is bounded independently of $\varepsilon$ because $\delta_1=0$.  We then end up with a CFL-type condition of the form $\Delta t \le C\Delta x^2$. Concrete results for specific schemes are given after the proof. Similarly, the number $K$ of inner steps can be bounded independently of $\varepsilon$ using the fact that $\log(\varepsilon) \to -\infty$ as $\varepsilon$ tends to $0$.

\begin{proof}[Proof of theorem~\ref{thm:stabcond}]
We know from theorem~\ref{thm:stab} that the stability regions of the projective forward Euler method are contained within those of the higher-order Runge--Kutta methods.  We therefore can safely choose the method parameters based on the stability conditions for the projective forward Euler method, which are given in equation~\eqref{eq:stab_pfe}.
The chosen method parameters $\delta t$, $\Delta t$ and $K$ need to be chosen to ensure that the eigenvalues in the region $\mathcal{D}_2$ (see corollary~\ref{cor:eig}) are contained in the region $\mathcal{D}_2^{PFE}$, and that the eigenvalue $\tau_{\delta t}$ is contained in the region $\mathcal{D}_1$.

First, we center the region $\mathcal{D}_2$ around the origin, resulting  in the requirement that
\begin{equation}\label{eq:Ddt}
	\delta t=\varepsilon^{\gamma+1}.
\end{equation}
Next, we need conditions on $\Delta t$ such that
$\tau_{\delta t}$ is contained within $\mathcal{D}_1^{PFE}$, i.e.
\[
1-2\dfrac{\varepsilon^{\gamma+1}}{\Delta t}\le\Re(\tau_{\delta t})\le 1, \qquad \left|\Im(\tau_{\delta t})\right|\le \dfrac{\varepsilon^{\gamma+1}}{\Delta t},
\]
where we have already used \eqref{eq:Ddt}. The second inequality on $\Re(\tau_{\delta t})$ is always satisfied. Using the expressions for the eigenvalues in corollary~\ref{cor:eig}, we obtain
\begin{equation}
\begin{cases}
1-2\dfrac{\delta t}{\Delta t}\leq 1+\eps\alpha+\eps^2\left(\Big\langle (\langle
\alpha\rangle -\alpha)^2-\beta^2 \Big\rangle +\delta_\gamma\Big\langle
\dfrac{\beta}{v}\Big\rangle^2\right)\\
\dfrac{\delta t}{\Delta t}\geq \Big|\eps \delta_\gamma
\Big\langle\dfrac{\beta}{v}\left\rangle+\eps^2\left[\delta_\gamma\Bigg\langle
\left(\Big\langle\dfrac{\beta}{v} \Big\rangle-\beta\right)(\langle\alpha
\rangle-\alpha) \Bigg\rangle +\delta_{\gamma-1}\Big\langle \dfrac{\beta}{v}
\Big\rangle\right]\right|
\end{cases}
\end{equation}
from which the condition in~\eqref{eq:Dt} is readily satisfied.


Finally, we have to choose $K$, the number of
small steps for the inner integrator, such that the eigenvalues in the region $\mathcal{D}_2$ are contained in the region $\mathcal{D}_2^{PFE}$.  From corollary \ref{cor:eig}, we
already know that, when $\delta t=\varepsilon^{\gamma+1}$, the radius $r$ of $\mathcal{D}_2$ is given as
\[
r = \dfrac{\varepsilon}{J}\max_{j\in\mathcal{J}}(|\alpha_j|+\beta_j|).
\]
Given that $\delta t$ and $\Delta t$ have already been fixed, the stability condition
\begin{eqnarray}
	r\leq \left( \dfrac{\delta t}{\Delta
	t}\right)^{1/K}=\left(\dfrac{\eps^{\gamma+1}}{\Delta t}\right)^{1/K}
\end{eqnarray}
results in a condition on $K$, which can be derived as $
K\geq \log\left( \eps^{\gamma+1}/\Delta t\right)/\log(r).
$

Using the conditions we have derived on $\delta t$ and $\Delta t$, we get:
\begin{equation}\label{eq:K-proof}
K\geq
\dfrac{\log\left(\varepsilon/b\right)}{\log\big(\varepsilon\,\max_{j}(|\alpha_j|+|\beta_j|)/J\big)},
\end{equation}
where $b$ is defined as in equation~\eqref{eq:b}. Remarking that~\eqref{eq:K-proof} is equivalent to~\eqref{eq:K} concludes the proof.
\end{proof}

We conclude with the application of the above stability
conditions for the specific combinations for the scaling and the spatial discretisation given in~\eqref{eq:derivx-par} and in~\eqref{eq:derivx-hyp}.


\begin{ex}[Hyperbolic scaling with third order upwind discretisation]
The hyperbolic case corresponds to $\gamma=0$, which implies $\delta t = \eps$.
Given the definitions~\eqref{eq:alpha-upwind} of $\alpha_j$ and $\beta_j$,
we obtain the following condition on $\Delta t$,
\begin{eqnarray*}
\Delta t&\leq& \min(\Delta t_1^{\textrm{max}},\Delta t_2^{\textrm{max}})\\
\Delta t^{\textrm{max}}_1&=&\left(\dfrac{2}{\left(\dfrac{
\langle |v|\rangle A(\zeta)}{6\Delta
x}\right)+\dfrac{\varepsilon}{36\Delta x^2}\left(1+\langle
v^2\rangle\right)B(\zeta)^2+\mathrm{Var}(|v|)^2A(\zeta)^2}\right)\\
\Delta t^{\textrm{max}}_2&=&\dfrac{1}{\dfrac{8\sin(\zeta)-\sin(2\zeta)}{6\Delta
x}+\varepsilon\left(\langle \alpha\beta\rangle- \langle\alpha\rangle\langle\beta\rangle \right)}
\end{eqnarray*}
It is clear that the order $\Oh(\varepsilon)$ term is positive in the denominator of $\Delta t^{\max}_1$, and $0$ in the denominator of $\Delta t^{\max}_2$, since $\langle \alpha\beta \rangle -\langle \alpha\rangle\langle\beta\rangle =0$. In the limit when $\varepsilon$ tends to $0$, we then obtain the following condition on $\Delta t$:
\begin{equation}
\Delta t\leq
\min\left(\dfrac{3\Delta x}{4\langle |v|\rangle},\dfrac{3\Delta x}{8}\right).
\end{equation}
We end up with a stability condition on the macroscopic time step
$\Delta t$ which is independent of $\varepsilon$, and that is of CFL-type for a hyperbolic partial differential equation.

To bound the number of inner steps $K$, we observe that, when $\varepsilon$ tends to $0$, the second term in~\eqref{eq:K} tends to $0$. Moreover, assuming $(1/J)\max_j(\left|\alpha_j\right|+\left|\beta_j\right|)\le 1$, the first term is bounded by $1$. With some algebraic manipulation, this leads to the condition that $K\ge 2$.
%
%
\end{ex}

\begin{ex}[Parabolic scaling with fourth order central discretisation]
The parabolic case corresponds to  $\gamma=1$, and therefore $\delta t=\eps^2$. For
the fourth order central discretisation, $\alpha_j$ and $\beta_j$ are given by~\eqref{eq:alpha-central}.

Substituting these expressions into~\eqref{eq:Dt} yields
\begin{eqnarray*}
\Delta t &\leq& \min (\mathrm{\Delta t^{\textrm{max}}_1,\Delta t^{\textrm{max}}_2})\\
\Delta t^{\textrm{max}}_1 &=& \dfrac{2\varepsilon(6\Delta x)^2}{\varepsilon
(8\sin(\zeta)-\sin(2\zeta))^2} \geq \dfrac{9\Delta x^2}{8\langle v^2\rangle}\\
\Delta t^{\textrm{max}}_2 &= &\dfrac{\varepsilon 6\Delta x}{\varepsilon
(8\sin(\zeta)-\sin(2\zeta)) }\geq \dfrac{3\Delta x}{4}
\end{eqnarray*}
Concerning the number $K$ of inner steps, a similar argument as above can be followed to show that the projective Runge--Kutta scheme will be stable provided that $K \ge 3$, see also \cite{lafitte2012}.
\end{ex}
\section{Consistency analysis} \label{sec:consist}

In this section we will prove that the PRK4 algorithm  is fourth order accurate
in space and time for a linear flux $A(u)=u$. \footnote{Remark that the following analysis can be done for
PRK schemes of any order.} First let us introduce some notations that will be
used throughout this section: for $k\in\{0,\hdots,K+1\}$,
\begin{itemize}
\item $t^{N,k}=N\Delta t+k\delta t$ is an intermediate time on the micro grid, as
 described in subsection \ref{sec:outer},
 \item $\widetilde{\partial^p f}^{N,k}$ denotes the evaluation of a $p$-th derivative of the exact solution of \eqref{eq:kin_eq_dimensionless} at time $t^{N,k}$,
 \item and $\tilde{u}^{N,k}=\langle \tilde{f}^{N,k}\rangle$ is the corresponding exact density,
 \item while $f^{N,k}$ is the numerical solution at time $t^{N,k}$
 resulting from the PRK4 scheme, starting from the exact solution
 $\tilde{f}^{N,K}$
 \item Similarly $u^{N,k}=\langle f^{N,k}\rangle$ is the corresponding numerical density function.
\end{itemize}
Therefore we will compute the truncation error $E^{N+1}$ at time
$t^{N+1}$
which is defined as:
\begin{equation}
E^{N+1}=\dfrac{\tilde{u}^{N+1}-u^{N+1}}{\Delta t},\label{eq_trunc}
\end{equation}
The expression for the truncation error the PRK4 scheme is:
\begin{equation}
E^{N+1} = \dfrac{\ut^{N+1}-u^{N,K+1}}{\Delta t}-\dfrac{\Delta t-(K+1)\delta
t}{\Delta t}\sum_{s=1}^Sb_s\dfrac{u^{N+c_s,K+1}-u^{N+c_s,K}}{\delta t}
\label{eq:truncv2},
\end{equation}
with, $\forall s\in\{1,\hdots,S\}$,
\begin{equation}
u^{N+c_s} = u^{N,K+1}+ (c_s\Delta t-(K+1)\delta
t)\sum_{l=1}^{s-1}\dfrac{ a_{sl}}{c_s}k_l,
\end{equation}
where, $\forall l\in\{1,\hdots,s-1\}$,
\begin{equation}
 k_l =\dfrac{u^{N+c_l,K+1}-u^{N+c_l,K}}{\delta t}.
\end{equation}
Furthermore,
the convergence error for the inner integrator reads:
\begin{equation}
e^{N,k}_f:=\dfrac{\tilde{f}^{N,k}-f^{N,k}}{\delta t}.\label{eq:trunc_inner}
\end{equation}
Recall that, since $\delta t=\eps^{\gamma+1}$,
\begin{equation}
f^{N,k+1}=\mathcal{S}_{\delta t}f^{N,k}=-\eps \deriv_{x,v}(f^{N,k})+
\Maxw u^{N,k}.\label{eq:step}
\end{equation}
\begin{rem}
 To stress the fact that $\Maxw$ is a linear operator, we omit the parenthesis
 of the argument in this section.
\end{rem}
Now we want to analyse the evolution of the truncation error of the inner integrator:
\begin{lem}
Suppose, we use an inner integrator which is accurate up to $p$-th order in
space and first order in time. Then, we also have
that $e^{N,K} =\Oh(\delta t)+\Oh(\varepsilon^{1-\gamma}\Delta x^p)$.\label{lem:trunc-inner}
\end{lem}
\begin{proof}
First, we analyse how the truncation error, defined in \eqref{eq:trunc_inner} evolves
after one extra step with the inner integrator.
Furthermore, we can expand the exact solution $\tilde{f}^{N,k+1}$ at time
$t^{N,k+1}$ around $t^{N,k}$ by using Taylor series:
\begin{align}
\tilde{f}^{N,k+1}&=\tilde{f}^{N,k}+\delta
t\,\widetilde{\partial_tf}^{N,k}+\Oh(\delta t^2)\\
&=\tilde{f}^{N,k}+\dfrac{\delta t}{\eps^\gamma}\left(-v\widetilde{\partial_x
   f}^{N,k}+\dfrac{-\tilde{f}^{N,k}+\Maxw{\langle\tilde{f}^{N,k}\rangle}}{\eps}
\right)+\Oh(\delta t^2)\\
&= \step \tilde{f}^{N,k} +\dfrac{1}{\eps^\gamma}(\deriv_{x,v}\tilde{f}^{N,k}-v\widetilde{\partial_xf}^{N,k})+\Oh(\delta t^2)\label{eq:ftilde}
\end{align}
Using \eqref{eq:step} and \eqref{eq:ftilde}, we get
\begin{equation}
 e^{N,k+1}= \step
 e^{N,k}+\dfrac{1}{\varepsilon^\gamma}\left(\deriv_{x,v}(\tilde{f}^{N,k}) -
\widetilde{v\partial_x f^{N,k}}\right)+\Oh(\delta t).
\end{equation}
Recall that
we suppose that the inner integrator is stable, and the assumption that the result at time $t^N$ is exact. This implies that we can write $e^{N,K+1}$ as:
\begin{equation}
	e^{N,K+1} =\dfrac{1}{\varepsilon^\gamma}\sum_{k=0}^{K-k}\step^k
	(\deriv_{x,v} \tilde{f}^{N,k}- v\widetilde{f}^{N,k}) +\Oh((K+1)\delta
	t).\label{eq:trunc_eNK1}
\end{equation}

To consider the above expression in more detail, we define: ${\Delta: f\mapsto
\deriv_{x,v}(\tilde{f})-\widetilde{v\partial_x f}}$ and recall that
$\deriv_{x,v}$ and $\Maxw$ are linear operators. Since
\begin{eqnarray}
\step (\Delta f)=\Maxw\langle \Delta f\rangle
-\varepsilon \deriv_{x,v}(\Delta f),
\end{eqnarray}
a simple recursion leads to, for all $k\geq 2$,
\begin{eqnarray*}
\step^k(\Delta f) = \Maxw\langle \Delta f\rangle
-\varepsilon \{\deriv_{x,v}\Maxw\langle \Delta
f\rangle+\Maxw\langle\Delta f\rangle\}+\Oh(\varepsilon^2).
\end{eqnarray*}
Now taking the mean value over velocity space yields:
\begin{equation}
\langle\Delta f \rangle =\langle \deriv_{x,v}(\tilde{u}+\varepsilon \tilde{g})-
\widetilde{v\partial_x(u+\varepsilon g)}\rangle=\varepsilon\langle \Delta
g\rangle=\varepsilon \Oh(\Delta x^p).
\end{equation}
The proof of the statement then follows by a simple substitution of the above
estimate into equation \eqref{eq:trunc_eNK1}.
\end{proof}
Now we can finally calculate the truncation error.
\begin{thm}[Truncation Error of PRK scheme]
 Consider a PRK scheme, that satisfies the assumptions \ref{ass:RK} and
 \eqref{eq:rk_convex} on the coefficients $a_{s,l}$. Then the truncation
 error $E^{N+1}$ of the scheme can be described by:
\begin{eqnarray*}
	E^{N+1}&=& \Bigg(\Delta t^2\mu^3\left(\dfrac{1}{6}-b^TA^2e\right)+\Delta
	t^3\mu^4\left(\dfrac{1}{24}-b^TA^3e\right)\Bigg)f^N+\Oh(\Delta t^4)\\
	&+& \Oh(\delta t)+\Oh(\varepsilon^{1-\gamma}\Delta x^p)+\delta
	t\Oh\left(\dfrac{\delta t+\varepsilon^{1-\gamma}\Delta x^p}{\Delta t}\right),
\end{eqnarray*}\label{thm:trunc}
where $\mu$ indicates the eigenvalue of the inner
integrator.
\end{thm}
\begin{proof}

	First we will derive a relation between the derivatives
	$k_s=\partial_t(f^{N+c_s})$ from the original Runge-Kutta scheme
	and the modified derivatives for our PRK scheme. So let us perform a Taylor expansion from $k_s$:
	\begin{align}
		k_s & = \dfrac{f^{N+c_s}+(K+1)\delta
		t\deriv_t(f^{N+c_s})+(K+1)^2\delta t^2\deriv_{tt}(f^{N+c_s})/2 -
		f^{N+c_s }-K\delta t \deriv_t (f^{N+c_s})-(K^2\delta
		t^2/2)\deriv_{tt}(f^{N+c_s})}{\delta t}\\
& = \deriv_t(f^{N+c_s}) +\dfrac{2k+1}{2}\delta t\deriv_{tt}(f^{N+c_s})
	+\Oh(\delta t^2).
\end{align}

Now we showed in lemma \ref{lem:trunc-inner}  that an application of the numerical derivative introduces an error
of order $\Oh(\varepsilon^{1-\gamma}\Delta x^p)$ with respect to the exact derivative
$\partial_t$ and hence we can write:
\begin{equation}
	k_s = \partial_t (f^{N+c_s}) + \dfrac{2K+1}{2}\delta t \partial_{tt}(f^N+c_s)
	+\Oh(\varepsilon^{1-\gamma}\Delta x^p).
\end{equation}
We proceed by substituting the expression for $f^{N+c_s}$ into the above
equation. This yields:
\begin{equation}
	k_s = \partial_t\left(f^{N,K+1} + (c_s\Delta t
	-(K+1)\delta t)\sum_{l=1}^{s-1}\dfrac{a_{sl}}{c_s}k_l\right) +\Oh(\delta
	t+\varepsilon^{1-\gamma}\Delta x^p).
\end{equation}
Of course, the latter can be further expanded as follows:
\begin{equation}
	k_s = \partial_t\left(f^N +(K+1)\delta t\deriv_t(f^N) + (c_s\Delta t
	-(K+1)\delta t) \sum_{l=1}\dfrac{a_{sl}}{c_s}k_l\right)+ \Oh(\delta
	t+\varepsilon^{1-\gamma}\Delta x^p),
\end{equation}
which is in turn equivalent to:
\begin{equation}
	k_s = \partial_t (f^N) + (K+1)\delta
	t\Big(\partial_{tt}(f^N)-\sum_{l=1}\dfrac{a_{sl}}{c_s}k_l\Big)+\Delta
	t\sum_{l=1}^{s-1}a_{sl}k_l+ (1+\delta t)\Oh(\delta
	t+\varepsilon ^{1-\gamma}\Delta x^p).
\end{equation}
Next, a combination of this result with the equation for $f^{N+1}$ gives rise to:
\begin{eqnarray*}
f^{N+1} &=& f^{N}+ (K+1)\delta t \deriv_t{f^N} + \dfrac{(K+1)^2}{2}\delta
t^2\deriv_{tt}(f^N) \\
&+& (\Delta t -(K+1)\delta t)\sum_{s=1}b_s\Big(\partial_t
(f^N) + (K+1)\delta t\Big(\partial_{tt}(f^N)-\sum_{l=1}\dfrac{a_{sl}}{c_s}k_l\Big)+\Delta
	t\sum_{l=1}^{s-1}a_{sl}k_l +(1+\delta t)\Oh(\delta t +\varepsilon^{1-\gamma}\Delta
	x^p)\Big).
\end{eqnarray*}
Then, we will proceed by splitting the above expression in an $\delta t$
-independent part and a part which depends on the time step of the inner
integrator. Now, we can apply the theorem about the order conditions of general
RK schemes (see \cite{Hairer1993}) to derive finally the expression for the
truncation error of the PRK scheme.
\end{proof}
\begin{ex}[Truncation error for PRK4]
As a direct consequence of theorem \ref{thm:trunc}, we find that the order of
accuracy of the PRK4 scheme is:
\begin{equation}
	E^{N+1}=\Oh(\Delta t^4)+\Oh(\delta t)+\Oh(\varepsilon^{1-\gamma}\Delta
	x^p)+\delta t\left(\dfrac{\delta t+\varepsilon^{1-\gamma}\Delta x^p}{\Delta t}\right),
\end{equation}
where we have used that $a_{21}=a_{32}=1/2$ and $a_{43}=1$. The other Runge
Kutta coefficients are zero. Moreover the scheme is consistent. This is the
scheme that we used throughout the numerical experiments.
\end{ex}

\section{Numerical experiments}\label{sec:exp}
In this section, we  illustrate the performance of the high-order
projective integration algorithm. In section~\ref{sec:num-lin}, we  first illustrate its consistency
properties and long term performance on a simple
linear kinetic equation. Afterwards, we will apply the scheme on some more
realistic applications: the Burgers' equation (section~\ref{sec:num-burgers}) and the semiconductor equation (section~\ref{sec:num-semi}).

In sections~\ref{sec:num-lin} and~\ref{sec:num-burgers}, we consider the velocity space $v_j,j\in
\{1\cdots J\}$  to be constructed using the zeroes of the Legendre polynomial of degree $J$; in section~\ref{sec:num-semi}, we use the zeroes of the Hermite polynomial of degree $J=20$.  All simulations are performed on the spatial domain $[-1,1]$. We choose an equidistant, constant in time mesh with cell centers $\Pi :=
\left\{x_0=-1+\Delta x/2,\ldots, 1-\Delta x/2\right\}$, and the fourth-order central spatial discretisation defined by~\eqref{eq:derivx-par}.


\subsection{Linear kinetic equation}\label{sec:num-lin}
We consider equation~\eqref{kin_eq} with $A(u) = u$ and periodic boundary conditions.
As an initial condition, we
take
\begin{equation}
	f(v,x,t) = \dfrac{e^{-v^2\sin(\pi x)/T}}{\sum_j w_j}, \qquad j=1, \ldots, J.
\end{equation}
To examine the truncation error (defined in
equation~\eqref{eq_trunc}), we perform a numerical simulation using a second and fourth order PRK algorithm, with Butcher tableaux in figure~\ref{fig:butcher}(right), using $\delta t=\varepsilon^2$, $K=3$, and $\Delta t=\num{1e-3}$, and $\Delta x=1\times 10^{-2}$; the number of outer PRK steps is defined by $(N+1)\Delta t=1$.  We perform the experiment for $\eps=1\times 10^{-2}$ and $\eps=1\cdot 10^{-3}$.
As the reference solution, we use a direct forward Euler simulation with  $\delta t=\varepsilon^3$.
The results are shown in figure~\ref{fig:truncerr}. One observes that the
truncation error behaves as $\Oh(\Delta t^4)$, resp., $\Oh(\Delta t^2)$, for large $\Delta t$ until, for sufficiently small $\Delta t$, a plateau is reached, at which the contribution of the inner integrator to the truncation error, which is $\Oh(\varepsilon^2)$, becomes dominant.
\begin{figure}[!h]
\centering
\subfigure{
 \includegraphics[width=0.40\textwidth]{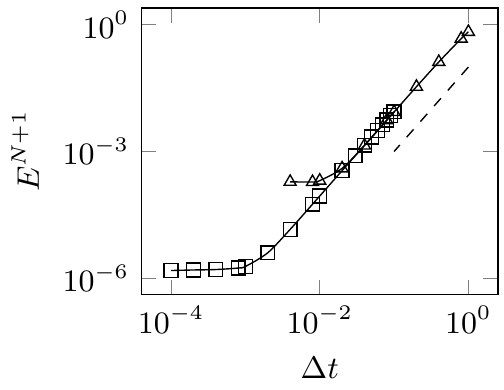}
}
\subfigure{
 \includegraphics[width=0.40\textwidth]{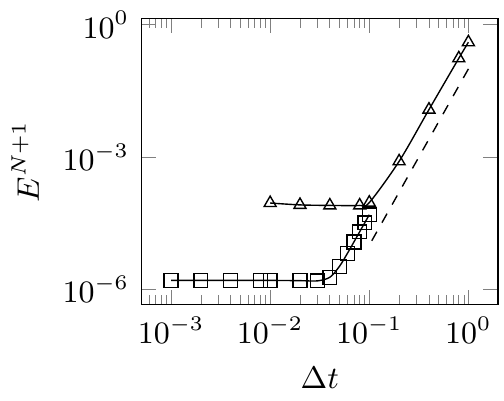}
}
	\caption{Truncation error of $u$ with PRK4 as a function of $\Delta t$ for
	$K=3, \Delta x=\num{1e-2}$ using $\varepsilon=\num{1e-3}$ (squares) and
	$\varepsilon=\num{1e-2}$ (triangles).}\label{fig:truncerr}
\end{figure}

Next, we compare the long-time simulation results of the PRK scheme with both a full microscopic
simulation and a simulation of the limiting macroscopic
equation~\eqref{eq:diff_limitpar}. We consider $\eps=1\times 10^{-3}$, and choose a
fourth-order PRK scheme with $K=3$, $\delta t=\eps^2$, $\Delta x=1\times
10^{-1}$ and $\Delta t=1\times 10^{-3}$.  The full microscopic simulation is
performed using the same inner integrator with time step $\delta t =\eps^3$, whereas the limiting macroscopic equation is simulated using the same spatial discretization and the corresponding direct Runge--Kutta method of order $4$.
The results are shown in figure~\ref{fig:longterm_lin}. We observe that the PRK4 algorithm is visually as
accurate as the full microscopic simulation, while requiring a computational
effort that is only $1/1000$  of the full microscopic simulation. Moreover, the
projective scheme appears to be able to capture the kinetic behaviour that is lost in the macroscopic limiting equation.
\begin{figure}[h]
\centering
\subfigure{
\includegraphics[width=0.40\textwidth, height=0.35\textwidth]{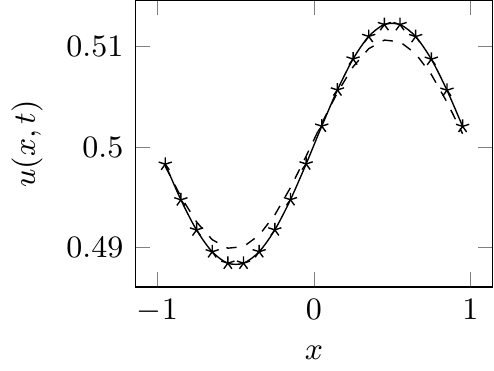}
}
\subfigure{
 \includegraphics[width=0.40\textwidth,
 height=0.35\textwidth]{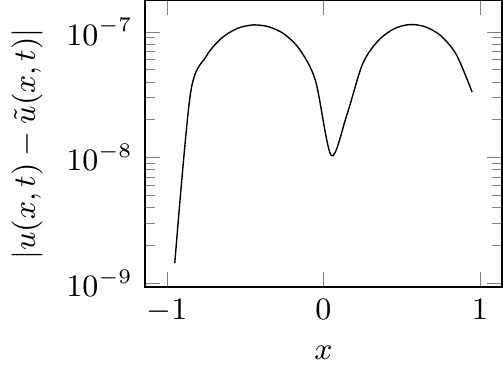}
%
}
\caption{Long term results for the linear kinetic equations. Left: density at
time $1$, ${\Delta t=\num{1e-3}}$, $K=3$, $\Delta x=\num{1e-1}$ with PRK4 (stars);
microscopic evolution with $\delta t=\varepsilon^3$ (solid line); results obtaind using the limiting equation (dashed). Right: absolute error of PRK results with respect to full microscopic simulation. }
\label{fig:longterm_lin}
\end{figure}

\subsection{Viscous Burgers' equation}\label{sec:num-burgers}

Let us now consider the viscous Burgers' equation, i.e., equation~\eqref{kin_eq} with $A(u) = u^2$, using Neumann boundary conditions
\[\partial_x f(-1,v,t)=
\partial_xf(1,v,t)=0),\]
and the initial condition:
\begin{equation}\label{eq:ic}
		f(x,v_j,t)=\dfrac{1}{\sum_j w_j}\exp(-v_j^2/T)\exp(-x^2/0.1).
\end{equation}
We again perform a fourth order PRK simulation using
$K=3$, $\delta t=\eps^2$, $\Delta x=1\times 10^{-1}$ and $\Delta t=1\times 10^{-3}$.  As a reference solution, we perform a full microscopic simulation using the same inner integrator with time step $\delta t =\eps^3$.
The results are shown in figure~\ref{fig:burgers}(left) at various instances in time. On the right, the error with respect to the reference soluton is shown. We clearly observe that the projective integration method is also very accurate in this case.
 \begin{figure}
 \centering
 \subfigure{
    \includegraphics[width=0.40\textwidth,height=0.35\textwidth]{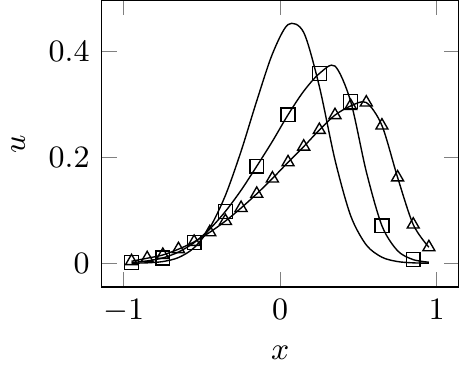}
 }
\subfigure{
 \includegraphics[width=0.40\textwidth,height=0.35\textwidth]{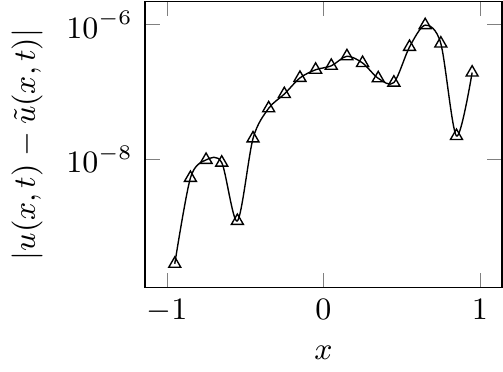}
%
 }
 	\caption{Left: Evolution of the Viscous Burgers equation as a result of the
 	PRK4 algorithm for $N=100$  (solid line, no marks), $N=500$ (squares) and
 	$N=1000$ (triangles). Right: Error at  $N=1000$. Parameters:  $\varepsilon=\num{1e-3}, \Delta t
 	= \num{1e-3}, K=3, \Delta x=\num{1e-1}$. }
 	\label{fig:burgers}
 \end{figure}

\subsection{Semiconductor equation}\label{sec:num-semi}

Finally, we illustrate the PRK method for the semiconductor equation~\eqref{eq:semi} with $\eps=1\times 10^{-3}$.

To discretise the partial derivative $\partial_v$, we also use a
second order finite difference scheme, taking into account that the chosen velocities, because they
are the zeroes of the Hermite polynomials, are not equidistant.  As an initial condition, we again choose~\eqref{eq:ic}
and we apply no-flux boundary conditions for both the velocity
and spatial variables. For the potential $\Phi$, we applied Dirichlet boundary
conditions, $\Phi(-1,t)=2.0$ and $\Phi(1,t)=0$, causing an advective movement to
the left.

\begin{figure}[h]
	\centering
		\subfigure{
		\includegraphics[width=0.40\textwidth,height=0.35\textwidth]{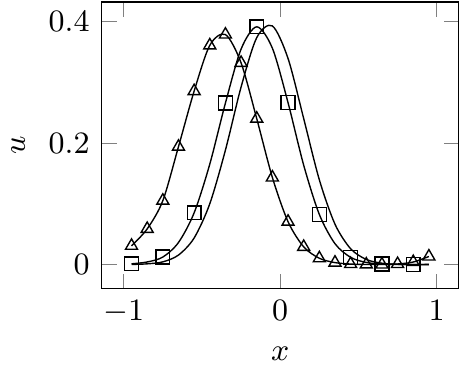}
		}
		\subfigure{
		\includegraphics[width=0.40\textwidth,height=0.35\textwidth]{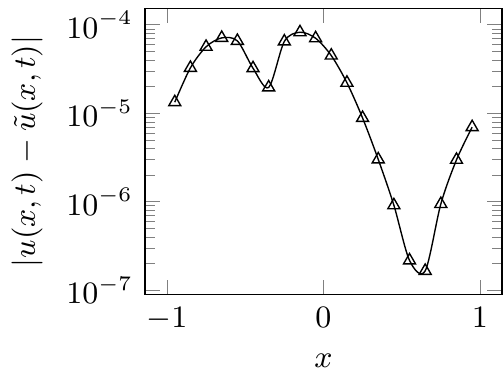}
		\tikzsetnextfilename{err_semiconductor}
			}
		\caption{Evolution of the density calculated with the projective integration
		algorithm after $N=100$ (no marks), $N=200$ (squares), $N=500$ (triangles).
		Parameters of the simulation: $\varepsilon=\num{1e-3}$,$\Delta t=\num{1e-3}$ ,
		$\Delta x=\num{1e-1}$ , $\Phi(-1,t)=-2.0$,$\Phi(1,t)=0.0$,   $T=\num{1e-2}$.}\label{fig:semi}
\end{figure}
As before, we perform a fourth order PRK simulation using
$K=3$, $\delta t=\eps^2$, $\Delta x=1\times 10^{-1}$ and $\Delta t=1\times 10^{-3}$, as well as a microscopic reference solution using the same inner integrator with time step $\delta t =\eps^3$.
The results are shown in figure~\ref{fig:semi}(left) at various instances in time. On the right, the error with respect to the reference soluton is shown. We clearly observe that the projective integration method is also very accurate in this case.
\section{Conclusions}\label{sec:concl}
We investigated a high-order, fully explicit, asymptotic-preser-ving scheme for a kinetic equation with linear relaxation, both in the hydrodynamic and diffusive scalings in which a hyperbolic, resp.~parabolic, limiting equation exists. The scheme first takes a few small (inner) steps with a simple, explicit method (such a direct forward Euler) to damp out stiff components of the solution and estimate the time derivative of the slow components. These estimated time derivatives are then used in an (outer) Runge--Kutta method of arbitrary order.  We showed that, with an appropriate choice of inner step size, the time-step restriction on the outer time step is similar to the stability condition for the limiting macroscopic equations. Moreover, the number of inner time steps is also independent of the scaling parameter.  We analyzed stability and consistency, and illustrated with numerical results.

We conclude by pointing out the current limitations of the method, and some suggestions for future work. The asymptotic-preserving nature of the scheme is due to the presence of a single relaxation time in the linear relaxation collision operator, and relies on an appropriate choice of the inner time step, which has to satisfy $\delta t=\varepsilon^{\gamma+1}$.   When multiple relaxation times are present, one should expect the number of time steps to be chosen as  $K\sim \log(1/\varepsilon)$ \cite{gear:1091}. In such situations, it might be of interest to study schemes in which a sequence of inner steps is taken, each commensurate with one of the relaxation time scales, as is proposed in \cite{logg}.  A second direction of further investigation would be to look at problems in which hydrodynamic and diffusive regimes are present simultaneously in different parts of the spatial domain.  Many efforts have been done for (semi-)implicit asymptotic-preserving schemes (some of them cited in the introduction); a number of these techniques (such as an a priori modeling of boundary layers \cite{golse,vignal}) can be readily applied in conjunction with the projective integration method proposed here.

\bibliographystyle{abbrv}
\bibliography{references.bib}

\appendix

\section{Parametrization of stability regions}

We need to derive the expressions that are given in Proposition~\ref{thm:stab_par}.
Let us start from the projective Runge--Kutta method as applied to the linear
test equation, i.e., equation~\eqref{eq:rk_ks}, which we now write as
\begin{align}
	\begin{cases}
k_1 &=:\kappa_1(\tau) y^n =\dfrac{\tau^{K+1}-\tau^{K}}{\delta t}y^n \\
k_s &=:\kappa_s(\tau) y^n= \dfrac{\tau^{K+1}-\tau^{K}}{\delta
t}F_s(\tau)y^n, \qquad 2 \le s \le S,\\
y^{n+1} &=: \sigma(\tau) y^n= \left(\tau^{K+1} + (M\delta t)\sum_{s=1}^{S}b_s
\kappa_s\right)y^n ,
\end{cases}
\label{eq:rk_ks}
\end{align}
with $F_s(\tau)=\left(\tau^{K+1}+(M_s\delta
t)\sum_{l=1}^{s-1}\dfrac{a_{s,l}}{c_s}\kappa_l\right)$.
The stability boundary of the projective Runge--Kutta method is then given by all values of $\tau$ such that
\begin{align}
\|\sigma(\tau)\|=\left|\tau^{K+1} + (M\delta t)\sum_{s=1}^{S}b_s
\kappa_s\right|=1.\label{eq:app_stab_cond}
\end{align}
For small values of $z=\delta t/\Delta t$, the stability region consists of two parts, one part close to the origin, and one part close to 1.
To locate these stability boundaries for small $z$, we proceed in two steps:
\begin{enumerate}[(i)]
	\item We notice that $\sigma(\tau)$ depends on $\kappa_s(\tau)$, while the $\kappa_s(\tau)$ themselves are recursively defined.  We therefore first obtain an explicit formula for each of the quantities $\kappa_s(\tau)$, such that we have an explicit formula for $\sigma(\tau)$ as a function of $\tau$.
	\item Next, for each of the stability regions, we perform an asymptotic expansion of $\tau$ as a function of $z$, and impose the condition~\eqref{eq:app_stab_cond}.
\end{enumerate}

Before proceeding with the derivation, we introduce some additional notation.  We will denote by $\bar{A}$ the matrix
	  \[
	\bar{A}=\bar{C}A=\mathrm{\diag}\left(M_1/c_1,M_2,\ldots,M_s/c_s\right)A
	\]
	 where $A$ denotes the matrix of RK-coefficients corresponding to a general $S$-stage
	  Runge--Kutta scheme. We also introduced the vectors $e_1$ and $e_r$ which are
	  defined as $(1,0,\ldots)^T$ and $(0,1,1,\ldots)^T$ respectively.

\paragraph{Step (i): Derivation of expression for $F_s(\tau)$} Let us now first derive an expression for $F_s(\tau)$.
We will show that
\begin{equation}\label{eq:Fstau}
	  \begin{cases}
	  F_s(\tau) =1 \qquad s=1,\\
	F_s(\tau)=
	\sum_{j=0}^{s-1}(\tau^K(\tau-1))^j((\bar{A}^j)(e_1+\tau^{K+1}e_r))_s\qquad

	\forall  2\le s \le S.

	\end{cases}
\end{equation}
By definition,
	we know that $F_s$ can be written as:
	\begin{equation}
		F_s(\tau) =\tau^{K+1}+\sum_{l=1}^{s-1}\bar{a}_{sl}F_l(\tau^{K+1}-\tau^K),
	\end{equation}
	where we have introduced $\bar{a}_{sl}=M_sa_{sl}/c_s$ to avoid notational
	complexity. Then we can derive equation ~\eqref{eq:Fstau} by induction.
	\begin{myenv}
		\textbf{Base step ($s=2$).} By definition the following holds:
		\begin{equation}
		F_2(\tau)= \tau^{K+1} +a_{21}(\tau^{K+1}-\tau^{K}).
		\end{equation}
		Now remark that $(\mathcal{I}e_1)_2 =0$, while $(\mathcal{I}e_r)_2=1.$ Then
		the statement follows from a simple substitution and rearranging the terms.\\
		\textbf{Induction step.} We impose that equation ~\eqref{eq:Fstau} is valid
		for all $l=1,\ldots s-1$ as induction hypothesis.  So, let's consider
		$F_s(\tau)$ in more detail:
	\begin{eqnarray}
		F_s(\tau) &=& \tau^{K+1} +\sum_{l=1}^{s-1}\bar{a}_{sl}
		\sum_{j=0}^{l-1}\left(\tau^K(\tau-1)\right)^j\left((\bar{A}^je_1)_l+(\bar{A}^{j}e_r)_l\tau^{K+1}\right)\tau^{K}(\tau-1)\\
		&=&
		\tau^{K+1} +
		\sum_{l=1}^{s-1}\bar{a}_{sl}\tau^{K}(\tau-1)+\left(\tau^K(\tau-1)\right)^2\sum_{l=2}^{s-1}\bar{a}_{sl}\left(\bar{A}e_1+\tau^{K+1}\bar{A}e_r\right)_l+\ldots\\
		&=& \tau^{K+1}+\sum (\tau^K(\tau-1))^l(\bar{A}^le_1+\tau^{K+1}\bar{A}^le_r)_s
	\end{eqnarray}

	\end{myenv}

	 and hence, the equation for the amplification factor $\sigma_S$
	reads:
	\begin{equation}
		\sigma_S(\theta)
		=\tau^{K+1}+M\sum_{s=1}^Sb_s\sum_{j=0}^{s-1}(\tau^K(\tau-1))^{j+1}\left((\bar{A}^j)(e_1+\tau^{K+1}e_r)\right)_s.
	\end{equation}
	\paragraph{Step (ii): Asymptotic expansion of $\tau$ for each of the two stability regions $\mathcal{R}_{1}^{\mathrm{PRK}}$ and $\mathcal{R}_{2}^{\mathrm{PRK}}$}

Let us first consider the region $\mathcal{R}_{1}^{\mathrm{PRK}}$, in which case $\tau$ is close to one.  We propose an asymptotic expansion for $\tau$ as a function of $z=\delta t/\Delta t$, and look for those values of $z$ for which~\eqref{eq:app_stab_cond} is satisfied. We expand
	$\tau(\theta)$ as follows:
	\begin{equation}
		\tau(\theta) = 1+C_1(\theta)z+C_2(\theta) z^2+h.o.t\qquad
		0\leq \theta \leq 2\pi
	\end{equation}
	By means of an application of the binomial theorem on both $\tau^{K+1}$ and
	$(\tau^K(\tau-1))^{j+1}$, we get
	\begin{eqnarray}
	\tau^{K+1}& =&
	1+C_1(K+1)z+(K+1)\left(C_2+\dfrac{(K)}{2}C_1^2\right)z^2+\Oh(z^3),\\
	(\tau^K(\tau-1))^{j}&=&
	C_1^{j}z^{j}+jC_1^{j-1}z^{j-1}\left((KC_1^2+C_2)z^2+\left(2KC_1C_2-\dfrac{1}{2}KC_1^3+\dfrac{1}{2}K^2C_1^3\right)z^3\right)+\Oh(z^4),\\
	\end{eqnarray}
	where we have momentarily suppressed dependence on $\theta$.
	Then, $F_s(z)$ can be expanded as
	\begin{eqnarray}
		F_s(z) &=& \sum_{j=0}^{s-1}
		(\tau^{K}(\tau-1))^j\left[\left(\dfrac{\bar{C}}{z}-(K+1)\right)^jA^j(e_1+e_r\tau^{K+1})\right]_s,\\
		&=&\sum_{j=0}^{s-1}\left(C_1^j+jC_1^{j-1}z(KC_1^2+C_2)\right)\left[\left(\bar{C}^j+j\bar{C}^{j-1}z(-(K+1))A^j\right)(e_1+e_r(1+C_1(K+1)z))\right]_s+\Oh(z^2)
	\end{eqnarray}
	where we have introduced $\bar{C}$ as $\mathrm{diag}(0,c_2,\ldots,c_s)$ and
	applied the binomial theorem to expand $\bar{A}^j$.
	This allows us to expand $\sigma_S$  as follows:

	\begin{equation}
		\sigma_S=
		1+\sum_{s=1}^Sb_s\sum_{l=1}^{s-1}C_1^j\left[\left(\bar{C}^jA^je_1\right)_s+\left(\bar{C}^jA^je_r\right)_s\right]+\Oh(z).\label{eq:sigma_R1}
	\end{equation}
	The coefficients $C_1,C_2$ can be determined by solving the equation $\sigma_S
=exp(\imag\theta)$ and matching powers of $(\delta t/\Delta t)$.

	In a similar way, this approach can be applied to the region
	$\mathcal{R}_{2,\mathrm{PRK}}$, where we expand $\tau(\theta)$ as:
	\begin{equation}
	\tau(\theta) =C_1'(\theta)z^{1/K} +
	C_2'(\theta)z^{2/K}+h.o.t.
	\end{equation}
We now derive the expressions
\begin{eqnarray}
		\tau^{K+1} &=& C_1'^{K+1}z^{K+1} +(K+1)C_1'^KC_2z^{K+2}+\Oh(z^{K+3}),\\
		\left(\tau^K(\tau-1)\right)^j &=&
		(-1)^jC_1'^{jK}z^{jK}+(-1)^{j-1}z^{jK+1}\left(C_1'^{jK+1}-jKC_1^{jK-1}C_2'\right),
	\end{eqnarray}
from which we obtain
\begin{eqnarray}
F_s &=&
\sum_{j=0}^{s-1}(-1)^{j}C_1'^{jK-1}\left(C_1'-z\left(C_1'^{2}-jKC_2'\right)\right)\bar{C}^{j-1}\left(\bar{C}-j(K+1)z\right)(A^je_1)_s+\Oh(z^2)\\
&=&\sum_{j=0}^{s-1}C_1'^{jK}\bar{C}^j(A^je_1)_s+\Oh(z),
\end{eqnarray}
which implies that $\sigma_S$ reads:

\begin{equation}
	\sigma_S=\sum_{s=1}^Sb_s\sum_{l=0}^{s-1}C_1'^{lK}\bar{C}^l(A^le_1)_s+\Oh(z).
\end{equation}
We conclude by
giving two concrete examples.

\begin{ex}[Parametrization of stability regions of PRK2]
Let us first consider the second order projective Runge--Kutta method PRK2.
We start with the region $\mathcal{R}_{2,\mathrm{PRK}}$. To this end, we have
to solve the following equation for $C_1'$
\begin{equation}
	\dfrac{1}{2}(C_1')^{2K}=e^{\imag\theta},
\end{equation}
which yields the roots:
\begin{equation}
C_1'=\sqrt{2}^{1/K}\exp(\imag(\theta/K+2j\pi/K))\quad j=1,\ldots, K-1.
\end{equation}
Hence, $\tau(\theta)$ can be written as
\begin{equation}
\tau=\sqrt{2}^{1/K}\exp{\imag(\theta+2j\pi/K)}\left(\dfrac{\delta t}{\Delta
t}\right)^{1/K}.
\end{equation}
For the stability region $\mathcal{R}_{1,\mathrm{PRK}}$, we use equation ~\eqref{eq:sigma_R1} to
determine the coefficient $C_1$,
\begin{equation}
	1-e^{\imag\theta}+C_1+\dfrac{1}{2}C_1^2 =0.
\end{equation}
which gives rise to the solution:
\begin{equation}
	C_1 = -1\pm \sqrt[4]{5-4\cos\theta}\arg(-1+2\cos\theta+2\imag\sin\theta).
\end{equation}
Thus, this part of the stability region is defined by:
\begin{equation}
	\tau =1-\left(\dfrac{\delta t}{\Delta t}\right)\pm
	\sqrt[4]{5-4\cos\theta}\arg(-1+2\cos\theta+2\imag\sin\theta)\left(\dfrac{\delta
	t}{\Delta t}\right)+\Oh\left(\left(\dfrac{\delta t}{\Delta t}\right)^2\right)
\end{equation}
\end{ex}

\begin{ex}[Parametrization of stability regions of PRK4]
The derivation is very similar to the second order case.
Now, we will determine the stability region around zero. So we will expand the
amplification factor of the inner integrator $\tau$ in powers of $(\tau/\Delta
t)^{1/K}$:
\begin{equation}
	\tau =C_1'\left(\dfrac{\delta t}{\Delta t}\right)^{1/K} + C_2'
	\left(\dfrac{\delta t}{\Delta t}\right)^{2/K},
\end{equation}
and substitute the latter into the stability polynomial equation. To find the
parametrization of the stability region, we set $\sigma=e^{\imag \theta}$ since
we are looking for the values of $\tau$ that result in a amplification factor
$|\sigma|=1$.  Next, we have to match the powers of $(\delta t/\Delta t)^{1/K}$
on both sides of the equations. This yields the following fourth order
polynomial in $C_1^K$:
\begin{equation}
\dfrac{1}{24}\left(C_1'^K\right)^4-\dfrac{1}{12}\left(C_1'^K\right)^3
+\dfrac{1}{6}\left(C_1'^K\right)^2-\dfrac{1}{6}C_1'^K-e^{\i\theta}=0
\end{equation}
This equation can be solved by using Ferrari's method \cite{Shmakov2011}. First,
we have to convert this polynomial into a so called \textit{depressed quartic},
by performing a change of variables: $C_1'^K = x+1/2$, which reduces the quartic
to :
\begin{equation}
p_1(x)=x^4+\dfrac{5}{2}x^2-x-\dfrac{19}{16}-24e^{\imag\theta}=0
\end{equation}
The latter can be factored into quadratic polynomials:
\begin{equation}
p_1(x) = (x^2+px+q)(x^2+rx+s)
\end{equation}
which results in the resolvent cubic polynomial in $P=p^2$:
\begin{equation}
P^3+5P^2+(11+96e^{\imag \theta})P-1=0
\end{equation}
This can be solved by performing again a change of variables : $P=t-5/3$ to
reduce the polynomial to a depressed cubic:
\begin{equation}
	t^3+\left(\dfrac{8}{3}+96e^{\imag\theta}\right)t-\dfrac{272}{27}-160e^{\imag\theta}=0,
\end{equation}
followed by Vi\'eta's substitution: $t=w-\dfrac{8/3+96e^{\imag\theta}}{3w}$
which finally yields with a quadratic polynomial in $w^3$:
\begin{equation}
	w^6+\left(-160e^{\imag\theta}-\dfrac{272}{27}\right)w^3-32768e^{3\imag\theta}-\dfrac{512}{729}-\dfrac{2048}{27}e^{\imag\theta}
	-\dfrac{8192}{3}e^{2\imag\theta}=0,
\end{equation}
which yields:
\begin{equation}
	w^3 =80e^{\imag\theta}+\dfrac{136}{27}\pm
	\dfrac{8}{9}\sqrt{41472e^{3\imag\theta}+11556e^{2\imag\theta}+1116e^{\imag\theta}+33},
\end{equation}
and hence, we can calculate a possible root of the sextic equation. This implies
that $w_1$ is a possible root:
\begin{equation}
w_1= \dfrac{2}{3}\sqrt[3]{270 e^{\imag \theta}+17+3\sqrt{41472e^{3\imag
\theta}+11556e^{2\imag\theta}+1116e^{\imag\theta}+33}}
\end{equation}
which implies that the following expression for $t$ is a root of equation:
\begin{equation}
	t= w_1 - \dfrac{8/3+96e^{\imag\theta}}{3w_1}.
\end{equation}
  Finally, we can calculate $C_1'$ by substituting the expression for $t$ back
  into the equation for $P$ and using the fact that $P=t-5/3$ and
  $C_1'^K=x+1/2$.
\begin{equation}
C_1' = \sqrt[K]{\dfrac{\pm\sqrt{P}
\pm\sqrt{P-2\left(\dfrac{5}{2}+P\pm\dfrac{1}{\sqrt{P}}\right)}}{2}+\dfrac{1}{2}}
\end{equation}
A similar procedure can be followed to derive an expression for $C_1$ to
determine the region $\mathcal{R}_{1,\mathrm{PRK}}$.
\end{ex}
\end{document}